\documentclass[sn-mathphys]{sn-jnl}

\jyear{2021}%

\theoremstyle{thmstyleone}%
\newtheorem{theorem}{Theorem}
\newtheorem{proposition}[theorem]{Proposition}%

\theoremstyle{thmstyletwo}%
\newtheorem{example}{Example}%
\newtheorem{remark}{Remark}%

\theoremstyle{thmstylethree}%
\newtheorem{definition}{Definition}%
\newcommand{\cred}[1]{{\color{black} #1}}
\newcommand{\corr}[1]{{\color{black} #1}}

\usepackage{mathdots}
\newtheorem{lemma}[theorem]{Lemma}
\usepackage{comment}
\usepackage{mathrsfs}
\usepackage{amssymb}
\begin{document}
	
		\title[A preconditioned MINRES method for evolutionary PDEs]{A sine transform based preconditioned MINRES method for all-at-once systems from constant and variable-coefficient evolutionary PDEs}
	
	
	
	\author*[1]{\fnm{Sean} \sur{Hon}}\email{seanyshon@hkbu.edu.hk}

\author[1]{\fnm{Po Yin} \sur{Fung}}\email{pyfung@hkbu.edu.hk}

\author[1]{\fnm{Jiamei} \sur{Dong}}\email{21482799@life.hkbu.edu.hk}

\author[2,3]{\fnm{Stefano} \sur{Serra-Capizzano}}\email{s.serracapizzano@uninsubria.it}

\affil*[1]{\orgdiv{Department of Mathematics}, \orgname{Hong Kong Baptist University}, \orgaddress{\street{Kowloon Tong}, \city{Hong Kong SAR},  \country{China}}}

\affil[2]{\orgdiv{Department of Science and High Technology}, \orgname{University of Insubria}, \orgaddress{ \city{Como}, \postcode{6152}, \state{Lombardy}, \country{Italy}}}

\affil[3]{\orgdiv{Department of Information Technology}, \orgname{Uppsala University}, \orgaddress{ \city{Uppsala}, \postcode{75008}, \state{Uppland}, \country{Sweden}}}
	
	
	\abstract{In this work, we propose a simple yet generic preconditioned Krylov subspace method for a large class of nonsymmetric block Toeplitz all-at-once systems arising from discretizing evolutionary partial differential equations. Namely, our main result is \cred{to propose two} novel symmetric positive definite \cred{preconditioners}, which can be efficiently diagonalized by the discrete sine transform matrix. More specifically, our approach is to first permute the original linear system to obtain a symmetric one, and subsequently develop desired \cred{preconditioners} based on the spectral symbol of the modified matrix. Then, we show that the eigenvalues of the preconditioned matrix sequences are clustered around $\pm 1$, which entails rapid convergence when the minimal residual method is devised. Alternatively, when the conjugate gradient method on \cred{the} normal equations is used, we show that our preconditioner is effective in the sense that the eigenvalues of the preconditioned matrix sequence are \cred{clustered around unity}. An extension of our proposed preconditioned method is given for high-order backward difference time discretization schemes, which can be applied on a wide range of time-dependent equations. Numerical examples are given, also in the variable-coefficient setting, to demonstrate the effectiveness of our proposed \cred{preconditioners}, which consistently outperforms an existing block circulant preconditioner discussed in the relevant literature.}

	
	
	
	\maketitle
	
	\section{Introduction}\label{sec1}
	
Over the past few years, there has been a growing interest on developing effective preconditioners for solving the all-at-once linear systems stemming from evolutionary partial differential equations (PDEs); see for instance \cred{\cite{Bertaccini2001,BertacciniNg2003,doi:10.1137/16M1062016,McDonald2017,HON2021113965,0x003ab34e,https://doi.org/10.1002/nla.2386,doi:10.1137/20M1316354,LIN2021110221,doi:10.1137/19M1309869,WU2021110076} and references therein}. Instead of solving PDEs in a sequential manner, this class of parallel-in-time (PinT) methods solves \cred{for} all unknowns simultaneously by constructing a large linear system which is composed of smaller linear systems \cred{at each time level}. Our proposed method in this work belongs to the diagonalization-based all-at-once methods \cite{10.1007/978-3-319-18827-0_50}, along with \cite{doi:10.1137/16M1062016,doi:10.1137/20M1316354}. Related PinT methods include space-time multigrid \cite{doi:10.1137/15M1046605,doi:10.1137/0916050}, multigrid reduction in time \cite{doi:10.1137/130944230,doi:10.1137/16M1074096}, and parareal method \cite{LIONS2001661,doi:10.1137/05064607X}. For a survey on the development of these PinT methods, we refer the readers to \cite{10.1007/978-3-319-23321-5_3} and the references therein.


As a model problem, {we consider} the following initial-boundary \cred{(value)} problem for the heat equation
\begin{equation}\label{eqn:heat}
	\left\{
	\begin{array}{lc}
		\frac{\partial u(x,t)}{\partial t} = \cred{\nabla \cdot (a({x})\nabla u(x,t))}+ f(x,t), & (x,t)\in \Omega \times (0,T], \\
		u=g, & (x,t)\in \partial \Omega \times (0,T], \\
		u(x,0)=u_0(x), & x \in \Omega.
	\end{array}
	\right.\,
\end{equation}
where $\Omega$ is open, $\partial \Omega$ \cred{denotes the boundary} of $\Omega$, and $a, f,g,u_0$ are all given functions.

After discretizing the spatial domain with a mesh and using the simple two-level $\theta$-method for time discretization with $n$ time-steps of size $\tau = \frac{T}{n}$, we get
\[M_m\bigg(\frac{\mathbf{u}_m^{(k+1)}-\mathbf{u}_m^{(k)}}{\tau}\bigg) =  -K_m\big( \theta  \mathbf{u}_m^{(k+1)} +  (1-\theta)  \mathbf{u}_m^{(k)}\big) + \theta \mathbf{f}_m^{(k+1)} + (1-\theta) \mathbf{f}_m^{(k)},
\]
where $M_m, K_m \in \mathbb{R}^{m\times m}$ are the mass matrix and the stiffness matrix approximating \cred{$-\nabla \cdot (a({x})\nabla  )$}, respectively, in the standard finite element terminology, $\mathbf{u}_m^{(k)}=[u_1^{(k)}, \dots, u_m^{(k)}]^{T},$ $\mathbf{f}_m^{(k)}=[f_1^{(k)}, \dots, f_m^{(k)}]^{T},$ and $\theta\in[0,1]$. The forward Euler method corresponds to the choice $\theta=0$, the backward Euler method to $\theta=1$, while the Crank-Nicolson method
is associated with $\theta=1/2$.

The discrete equations to be solved are described below
\begin{equation}\nonumber
	(M_m+\theta \tau K_m) \mathbf{u}_m^{(k+1)} = (M_m- (1-\theta) \tau K_m) \mathbf{u}_m^{(k)} + \theta \tau \mathbf{f}_m^{(k+1)} + (1-\theta) \tau \mathbf{f}_m^{(k)}.
\end{equation}
{Instead of solving the above equations for $\mathbf{u}_m^{(k)}$ sequentially, for $k=1,2,\ldots, n$, we solve the following equivalent all-at-once system}
\[\mathcal{T}_H
\underbrace{\begin{bmatrix}{}\mathbf{u}_m^{(1)}\\ \mathbf{u}_m^{(2)}\\ \mathbf{u}_m^{(3)}\\ \vdots \\ \mathbf{u}_m^{(n)}
\end{bmatrix}}_{\corr{=:\mathbf{u}}}
=
\underbrace
{\begin{bmatrix}{} \theta \tau \mathbf{f}_m^{(1)}+ (1-\theta) \tau \mathbf{f}_m^{(0)} + (M_m- (1-\theta) \tau K_m) \mathbf{u}_m^{(0)}\\
		\theta \tau \mathbf{f}_m^{(2)}+ (1-\theta) \tau \mathbf{f}_m^{(1)}\\
		\theta \tau \mathbf{f}_m^{(3)}+ (1-\theta) \tau \mathbf{f}_m^{(2)}\\
		\vdots \\
		\theta \tau \mathbf{f}_m^{(n)}+ (1-\theta) \tau \mathbf{f}_m^{(n-1)}
\end{bmatrix} }_{\corr{=:\mathbf{f}_H}},
\]where
\begin{equation}\label{theta}
	\mathcal{T}_H=\begin{bmatrix}
		M_m+\theta \tau K_m &   &  &  \\
		-M_m+ (1-\theta) \tau K_m  & \ddots   & &  \\
		&   \ddots & \ddots &  \\
		&    & \ddots &  \ddots\\
		&   & -M_m+ (1-\theta) \tau K_m & M_m+\theta \tau K_m
	\end{bmatrix}
\end{equation}
is a $mn$ by $mn$ nonsymmetric block Toeplitz matrix, which is generated by the matrix-valued function
\begin{equation}\label{eqn:symbol_h}
	g_{\theta}(x)=\underbrace{M_m+\theta \tau K_m}_{=:A_{(0)}}+\underbrace{(-M_m+ (1-\theta) \tau K_m)}_{=:A_{(1)}}e^{\textbf{i}x}.
\end{equation}

According to the notion of generating function reported in Section \ref{sec:prelim} (see also \cite{book-GLT-I} and the references there reported), it is noted that if $K_m$ and $M_m$ commute then $A_{(0)}$ and $A_{(1)}$ commute. If one uses a finite element formulation to discretize in space for (\ref{eqn:heat}), then $M_m$ and \cred{$K_m$} are simultaneously diagonalizable, provided that a uniform grid is used. Alternatively, if finite difference methods are used for (\ref{eqn:heat}), then $M_m$, being exactly the identity matrix $I_m$ in this case, always commutes with $K_m$ (see \cite[Section 3.1.1]{doi:10.1137/16M1062016} for {a discussion}).

Therefore, throughout this work, the following assumptions on both $M_{m}$ and $K_{m}$ are made, which are compatible with \cite[Section 2]{doi:10.1137/20M1316354}:
\begin{enumerate}
	\item Both $M_{m}$ and $K_{m}$ are real symmetric positive definite (SPD);
	\item Both $M_{m}$ and $K_{m}$ are sparse, i.e., they have only $\mathcal{O}(m)$ nonzero entries.
	\item \cred{$M_{m}$ and $K_{m}$ commute}.
\end{enumerate}

As already observed the commutation between $M_{m}$ and $K_{m}$ is obvious if $M_m$ is the identity matrix or in the constant coefficient setting. When $a(x)$ is not constant or there is also a dependency on time, that is $a\equiv a(x,t)$, \cred{or the discretization method is different (finite elements, finite volumes, IgA) or graded meshes are used, the commutation is no longer true algebraically, but it holds asymptotically for large enough $m$ and it can be proved via multilevel GLT tools (see \cite{book-GLT-II}). More in detail, by assuming $m=m(n)$ with $m(n)$ going to infinity as $n$ tend to infinity and so that $1=o(m(n))$ and
\begin{equation}\label{hp-c}
c=\lim_{h\rightarrow 0^+} \frac{\tau}{h^2} >0,
\end{equation}
in the case of a variable coefficient $a\equiv a(x,t)$, we obtain that the matrix sequence $\{h^2 K_m\}_m$ is spectrally distributed as $p(\psi) a(t,x)$ on $(\psi,x)\in [0,\pi,\Omega]$, in the sense of Definition \ref{block-distr}, since $\{h^2 K_m\}_m$ a real symmetric GLT matrix sequence with GLT symbol $p(\psi) a(t,x)$. The latter statement holds since we exploit a canonical decomposition of $h^2 K_m$ as the product of a diagonal sampling sequence and of Toeplitz matrix sequence plus a zero-distributed matrix sequence (see again \cite[Ch. 6, Sect. 7.3]{book-GLT-II}). Hence, under the assumption that finite differences are used and by exploiting the tensor structure of $\mathcal{T}_H$ in (\ref{theta}), the global matrix sequence $\{\mathcal{T}_H\}_m$ will be again a GLT matrix sequence with GLT symbol
\[
c\theta \, p(\psi_2) a(t,x) + c(1-\theta)\, a(t,x)e^{-\mathbf{i} \psi_1 }.
\] 
Still the same distribution is present if condition (\ref{hp-c}) is violated and $c=\infty$, but with a different scaling factor and it applies to the matrix sequence  $\{\frac{h^2}{\tau}\mathcal{T}_H\}_m$.

However, in the present setting $m$ is considered fixed that is $c=0$ and hence also the condition $1=o(m(n))$ is not satisfied. Here a matrix-valued distribution is deduced. When assuming that both $n$ and $m$ are allowed to tend to infinity, in the way indicated in the previous lines, the situation changes and the related type of analysis represents a research line with various cases to be considered in the future.
}

	Instead of directly solving (\ref{theta}), we consider the permuted linear system $\mathcal{Y} \mathcal{T}_H \mathbf{u} = \mathcal{Y}\mathbf{f}_H$, where
\begin{equation}\label{eqn:main_system}
	\mathcal{Y} \mathcal{T}_H
	=
	\begin{bmatrix}
		&  &   & A_{(1)} & A_{(0)}\\
		&   & \iddots & \iddots &  \\
		&  A_{(1)}  & A_{(0)}  & &  \\
		A_{(1)}  & A_{(0)}    & & & \\
		A_{(0)} &   &  &  & \\
	\end{bmatrix}.
\end{equation}
	Note that $\mathcal{Y} = Y_n \otimes I_m$ with $Y_n \in \mathbb{R}^{n \times n}$ being the anti-identity matrix or the flip matrix that is
$[Y_n]_{j,k}=1$ if and only if $j+k=n+1$ and $[Y_n]_{j,k}=0$ otherwise. Clearly, $\mathcal{Y} \mathcal{T}_H$ is symmetric since $L_m$ is assumed symmetric. Thus, we can then deploy a symmetric Krylov subspace method such as the minimal residual (MINRES) method and design an effective preconditioner for $\mathcal{Y} \mathcal{T}_H$, whose convergence depends only on the eigenvalues.

In \cite{doi:10.1137/16M1062016}, a Strang block circulant preconditioner was proposed for $\mathcal{Y} \mathcal{T}_H$ and proved effective when $a(x)=1$. The main advantage of this iterative approach is that the inverse of the preconditioner can be computed using parallel fast Fourier transforms over a number of processors. However, the performance of such a preconditioner can be unsatisfactory, as will be shown in the numerical examples in Section \ref{sec:numerical}. We remark that such an unsatisfactory preconditioning effect of \corr{circulant matrices \cite{Strang1986,doi:10.1137/0909051}} for certain ill-conditioned Toeplitz systems was discussed in \cite{Hon_SC_Wathen}.

Thus, instead of using block circulant preconditioners, we propose in this work the following SPD preconditioner for $\mathcal{Y} \mathcal{T}_H$, that can be diagonalized by the discrete sine transform matrix
\begin{eqnarray}\label{def:matrix_P}
	\mathcal{P}_H=\sqrt{\begin{bmatrix}
			A_{(0)}^2 + A_{(1)}^2 & A_{(0)}A_{(1)}  &  &  & \\
			A_{(0)}A_{(1)}  & A_{(0)}^2 + A_{(1)}^2  &  \ddots & & \\
			&   \ddots & \ddots & \ddots &  \\
			&   & \ddots & \ddots & A_{(0)}A_{(1)}  \\
			&  &   & A_{(0)}A_{(1)}  & A_{(0)}^2 + A_{(1)}^2
	\end{bmatrix}},
\end{eqnarray}
which is constructed based on approximating the asymptotic eigenvalue distribution of $\mathcal{Y} \mathcal{T}_H$. Similar to the block circulant preconditioner, the computation of $\mathcal{P}_H^{-1}$ can also be parallelizable over $n$ processors. We refer readers to \cite{Bini1990} for a detailed discussion on such a parallel implementation.

We remind that in a pure Toeplitz setting the algebra of matrices diagonalized by the (real) orthogonal sine transform (i.e., the so-called $\tau$ algebra \cite{BC83}) was extensively used in the 90s (see \cite{D1,D2,DS,Slinear99} and references therein), while recently the approach based on $\tau$ preconditioning gained further attention in the context of approximated fractional differential equations (see \cite{BEV21,DMS16} and references therein). Here, we propose its use in a different setting where the Toeplitz structure is nonsymmetric. We recall that the $\tau$ preconditioning is superior when compared to the circulant preconditioning, when real symmetric Toeplitz-like structures are considered, as discussed and theoretically motivated in \cite{DS,Slinear99}.

In particular, we will show that the eigenvalues of the preconditioned matrix sequence $\{\mathcal{P}_H^{-1}\mathcal{Y} \mathcal{T}_H\}_n$ are clustered around $\pm 1$, which leads to fast convergence when MINRES is employed.

\cred{Yet, $\mathcal{P}_H$ requires fast diagonalizability of $M_{m}$ and $K_{m}$ in order to be efficiently implemented. When such diagonalizability is not available, we further propose the following preconditioner $\mathcal{P}_{\theta}$ as a modification of $\mathcal{P}_H$,
\begin{eqnarray}\label{def:matrix_mod_P}
	\mathcal{P}_{\theta} = \mathcal{H} \otimes M_{m} + \mathcal{H}_{\theta} \otimes \tau K_{m}
\end{eqnarray}
where
\begin{eqnarray*}
\mathcal{H} = \sqrt{\begin{bmatrix}
2 & -1 &  &  &  \\
-1 & 2 & -1 &  &  \\
  & \ddots & \ddots & \ddots &  \\
 &  & \ddots & \ddots & -1 \\
 &  &  & -1 & 2
\end{bmatrix}}
\end{eqnarray*}
and
\begin{eqnarray*}
\mathcal{H}_{\theta} = \sqrt{\begin{bmatrix}
\theta^2+(1-\theta)^2 & \theta(1-\theta) &   &  &  \\
\theta(1-\theta) & \theta^2+(1-\theta)^2 & \ddots &  &  \\
  & \ddots & \ddots & \ddots &  \\
 &  & \ddots & \ddots & \theta(1-\theta) \\
 &  &  & \theta(1-\theta) & \theta^2+(1-\theta)^2
\end{bmatrix}}. 
\end{eqnarray*} Its derivation and implementation will be discussed in Section \ref{subsec:MP}.
} 

In general, as mentioned in \cite[Chapter 6]{ANU:9672992}, the convergence study of preconditioning strategies for nonsymmetric problems is heuristic, since descriptive convergence bounds are usually not available for the generalized minimal residual (GMRES) method or any of the other applicable nonsymmetric Krylov subspace iterative methods.

The paper is organized as follows. In Section \ref{sec:prelim}, we review some preliminary results on block Toeplitz matrices and matrix-valued functions. We present our main results in Section \ref{sec:main} where our proposed preconditioner and its effectiveness are provided. An extension of our preconditioning method is given in Section \ref{sec:extension} {in the setting of high-order in time discretization schemes.} Numerical tests are given in Section \ref{sec:numerical} to support {the analysis of our proposed preconditioner. Finally in Section \ref{sec:conclusions} we draw conclusions and list a few open problems, including the challenging case of variable-coefficients when both the sizes $m$ and $n$ tend to infinity.}

	\section{Preliminaries}
\label{sec:prelim}

In the following subsections, we provide some useful background knowledge about block Toeplitz matrices and matrix functions.

	\subsection{Block Toeplitz matrices}

We let $L^1([-\pi,\pi],\mathbb{C}^{m \times m})$ be the Banach space of all matrix-valued functions that are Lebesgue integrable over $[-\pi,\pi]$. The $L^1$-norm induced by the trace norm over $\mathbb{C}^{m \times m}$ is
\[
\|f\|_{L^1} = \frac{1}{2\pi}\int_{-\pi}^{\pi} \|f(\psi)\|_{\text{tr}}\,d \psi < \infty,
\]
where $\|A_n\|_{\text{tr}}:=\sum_{j=1}^{n}\sigma_{j}(A_n)$ denotes the trace norm of $A_n\in \mathbb{C}^{n \times n}$. The block Toeplitz matrix generated by $f \in L^1([-\pi,\pi],\mathbb{C}^{m \times m})$ is denoted by $T_{(n,m)}[f]$, namely
\begin{eqnarray}\label{eqn:det_blockT}
	{T}_{(n,m)}[f]&=&
	\begin{bmatrix}
		A_{(0)} & A_{(-1)} & \cdots & A_{(-n+1)} \\
		A_{(1)} & \ddots & \ddots & \vdots \\
		\vdots & \ddots & \ddots  & A_{(-1)} \\
		A_{(n-1)} & \cdots & A_{(1)} & A_{(0) }
	\end{bmatrix} \in \mathbb{C}^{mn \times mn},
\end{eqnarray}
where the Fourier coefficients of $f$ are
\[ A_{(k)}=\frac{1}{2\pi} \int_{-\pi}^{\pi} f(\psi) e^{-\mathbf{i} k \psi } \,d\psi\in \mathbb{C}^{m \times m},\qquad k=0,\pm1,\pm2,\dots.
\] 
The function $f$ is called the \emph{generating function} or \emph{symbol} of $T_{({n},m)}[f]$. For thorough discussions on the related properties of (multilevel) block Toeplitz matrices, we refer the readers to \cite{MR2108963,Chan:1996:CGM:240441.240445,book-GLT-I} and references therein.

	Before discussing the asymptotic singular value or spectral distribution of $T_{(n,m)}[f]$ associated with $f$, which is crucial to develop our preconditioning theory, we introduce the following definition.
	
	Let $C_c(\mathbb{C})$ (or $C_c(\mathbb{R})$) be the space of complex-valued continuous functions defined on $\mathbb{C}$ (or $\mathbb{R}$) with bounded support.

	\begin{definition}\cite{Ferrari2019}\label{block-distr}
		Let $\{A_{(n,s)}\}_n$ be a sequence of $sn$-by-$sn$ \cred{matrices with $s$ fixed positive integer}.
		\begin{enumerate}[1.]
			
			\item We say that $\{A_{(n,s)}\}_n$ has an asymptotic singular value distribution described by an $s$-by-$s$ matrix-valued $f$ if
			\[
			\lim_{n\to \infty}\frac{1}{sn}\sum_{j=1}^{ns}F(\sigma_j(A_{(n,s)}))=\frac{1}{2\pi}\int_{-\pi}^{\pi}\frac{1}{s} \sum_{j=1}^{s}F(\sigma_j(f(x))) \,dx,\quad \forall F \in C_c(\mathbb{R})
			\]
			and we write $\{A_{(n,s)}\}_n \sim_{\sigma} f$.
			
			\item We say that $\{A_{(n,s)}\}_n$ has an asymptotic spectral distribution described by an $s$-by-$s$ matrix-valued $f$ if
			\[
			\lim_{n\to \infty}\frac{1}{sn}\sum_{j=1}^{np}F(\lambda_j(A_{(n,s)}))=\frac{1}{2\pi}\int_{-\pi}^{\pi}\frac{1}{s} \sum_{j=1}^{s}F(\lambda_j(f(x))) \,dx,\quad \forall F \in C_c(\mathbb{C})
			\]
			and we write $\{A_{(n,s)}\}_n \sim_{\lambda} f$.
		\end{enumerate}

	\end{definition}

	Before presenting our main preconditioning results, we introduce the following notation. Given $D \subset \mathbb{R}^k$ with $0<\mu_k(D)<\infty$, we define $\widetilde D_p$ as $D\bigcup D_p$, where $p\in  \mathbb{R}^k$ and $D_p=p+D$, with the constraint that $D$ and $D_p$ have non-intersecting interior part, i.e., $D^\circ \bigcap  D_p^\circ  =\emptyset$. In this way, we have $\mu_k(\widetilde D_p)=2\mu_k(D)$. Given any $g$ defined over $D$, we define $\psi_g$ over $\widetilde D_p$ in the following way
	\begin{equation}\label{def-psi}
		\psi_g(x)=\left\{
		\begin{array}{cc}
			g(x), & x\in D, \\
			-g(x-p), & x\in D_p, \ x \notin D.
		\end{array}
		\right.\,
	\end{equation}
	
	\begin{theorem}\cite[Theorem 3.4]{Ferrari2019}{}\label{thm:main_hon-ter}
		Suppose $f \in L^1([-\pi,\pi], \mathbb{C}^{m \times m})$ has Hermitian Fourier coefficients. Let $T_{(n,m)}[f]\in \mathbb{C}^{mn \times mn}$ be the block Toeplitz matrix generated by $f$ and let $Y_{(n,m)}=Y_n \otimes I_m\in \mathbb{R}^{mn \times mn}$. Then
		\[
		\{Y_{(n,m)}T_{(n,m)}[f]\}_n \sim_{\lambda}  \psi_{\vert f \vert}, \qquad \vert f \vert=(f f^*)^{1/2},
		\]
		over the domain  $\widetilde D$ with $D=[0,2\pi]$ and $p=-2\pi$, where $\psi_{ \vert f \vert}$ is defined in (\ref{def-psi}). That is
		{\cred{
		\[
		\lim_{n\to \infty} \frac{1}{mn}\sum_{j=1}^{mn}F(\lambda_j(Y_{(n,m)}T_{(n,m)}[f]))=
		\frac{1}{4m\pi}\int_{-\pi}^{\pi} 
		\sum_{j=1}^{m} F(\lambda_j(\vert f \vert(x)))+F(-\lambda_j(\vert f \vert(x))) \,dx,
		\]
		}}
		where $\lambda_j(\vert f \vert(x))$, $j=1,2,\dots,m$, are the eigenvalue functions of $\vert f \vert$.
	\end{theorem}
	
	As for the symmetrized matrix $ \mathcal{Y} \mathcal{T}_H$ given in (\ref{eqn:main_system}), we can see by \cite{MazzaPestana2018, Ferrari2019} that their eigenvalues are distributed as $\pm \vert g_{\theta} \vert$. By Theorem \ref{thm:main_hon-ter}, $ \mathcal{Y} \mathcal{T}_H$ is always symmetric indefinite when $n$ is sufficiently large, which gives grounds for the use of MINRES in this work.

		\subsection{Matrix functions}
	
	In this subsection, several useful results on functions of matrices are presented.
	
	If $F$ is analytic on a simply connected open region of the complex plane containing the interval $[-1, 1]$, there exist ellipses with foci in $-1$ and $1$ such that $F$ is analytic in their interiors. Let $\alpha > 1$ and $\beta >0$ be the half axes of such an ellipse with $\sqrt{\alpha^2 - \beta^2}=1$. Such an ellipse, denoted by ${E}_{\mathcal{X}}$, is completely specified once the number $\mathcal{X}:= \alpha + \beta$ is known.
	
	\begin{theorem}[Bernstein's theorem]\cite[Theorem 2.1]{Benzi1999}\label{thm:Bernstein}
		Let the function $F$ be analytic in the interior of the ellipse $\mathbb{E}_{\mathcal{X}}$ with $\mathcal{X}>1$ and continuous on $\mathbb{E}_{\mathcal{X}}$. In addition, suppose $F(x)$ is real for real $x$. Then, the best approximation error
		\[
		E_k(F):=\textrm{inf}\{\| E - p \|_{\infty} : \textrm{deg}(p) \leq k\} \leq \frac{2M(\mathcal{X})}{\mathcal{X}^k(\mathcal{X}-1)},
		\]
		where $\textrm{deg}(p)$ denotes the degree of the polynomial $p(x)$ and
		\[
		\| F - p \|_{\infty} = \max_{-1 \leq x \leq 1} \vert F(x)-p(x) \vert, \quad M(\mathcal{X})=\max_{x \in \mathbb{E}_{\mathcal{X}}}\{ \vert F(x) \vert \}.
		\]
	\end{theorem}
	
	Let $A_n$ be an $n \times n$ symmetric matrix and let $[\lambda_{\min}, \lambda_{\max}]$ be the smallest interval containing $\sigma(A_n)$. If we introduce the linear affine
	function \[\psi(\lambda) = \frac{2 \lambda - (\lambda_{\min} + \lambda_{\max})}{\lambda_{\max} - \lambda_{\min}},\] then $\psi([\lambda_{\min},\lambda_{\max}])=[-1,1]$ and therefore the spectrum of the symmetric matrix \[B_n:=\psi(A_n)=\frac{2}{\lambda_{\max} - \lambda_{\min}}A_n - \frac{\lambda_{\min} + \lambda_{\max}}{\lambda_{\max} - \lambda_{\min}}I_n\] is contained in $[-1,1]$. Given a function $f$ analytic on a simply connected region containing $[\lambda_{\min}, \lambda_{\max}]$ and such that $f(\lambda)$ is real when $\lambda$ is real, the function $F = f \circ \psi^{-1}$ satisfies the assumptions of Bernstein's theorem.
	
	In the special case where $A_n$ is SPD and $f(x) = x^{-1/2}$ that are of our interest in this work, we apply Bernstein's result to the function
	\begin{equation}\label{def:function_F}
		F(x) = \frac{1}{ \sqrt{ \frac{(b - a)}{2}x  + \frac{ a+ b}{2} } },
	\end{equation}
	where $a=\lambda_{\min}(A_n), b=\lambda_{\max}(A_n)$, and $1 < \mathcal{X} < \frac{\sqrt{\kappa}+1}{\sqrt{\kappa}-1}$ with the spectral condition number of $A_n$ being $\kappa={b}/{a}.$

	\section{Main results}\label{sec:main}
	
	In this section, we provide our results on the proposed preconditioner $\mathcal{P}$ defined by (\ref{def:matrix_P}), whose design is based on the following spectral symbol
	\begin{eqnarray}\nonumber
		\vert g_{\theta}\vert^2  &=&  A_{(0)}^2 + A_{(1)}^2 + A_{(0)}A_{(1)} e^{-\mathbf{i}x}+A_{(0)}A_{(1)} e^{\mathbf{i}x}\\\label{eqn:symbol_p}
		&=&A_{(0)}^2 + A_{(1)}^2 + 2A_{(0)}A_{(1)}\cos{(x)}.
	\end{eqnarray} Due to the fact that $\mathcal{P}=\sqrt{T_{(n,m)}[\vert g_{\theta}\vert^2]}$, it can be regarded as an optimal preconditioner for $\mathcal{Y}\mathcal{T}$, in the sense of the spectral distribution theory developed in the work \cite{Ferrari2019}. We discuss mainly MINRES combined with our proposed preconditioner, and related issues such as implementations and convergence analysis are given as well.

	\subsection{Preconditioning for heat equations}

In this section, we provide a MINRES approach for $\mathcal{Y}\mathcal{T}_H\mathbf{u}=\mathcal{Y}\mathbf{f}_H$.

\subsubsection{Implementations}\label{sebsebsec:imple_MINRES}

We first discuss the fast computation of $\mathcal{Y}\mathcal{T}_H\mathbf{v}$ for any given vector $\mathbf{v}$. Since both $M_m$ and $K_m$ are sparse matrices, $\mathcal{T}_H$ is sparse. Hence, computing the matrix-vector product $\mathcal{T}_H\mathbf{v}$ only requires linear complexity of $\mathcal{O}(mn)$. Finally, computing $\mathcal{Y}\mathcal{T}_H\mathbf{v}$ needs the same complexity since $\mathcal{Y}$ simply imposes a reordering on the vector $\mathcal{T}_H\mathbf{v}$.  Alternatively, due to the fact that $\mathcal{T}_H$ itself is a block Toeplitz matrix, it is well-known that the product $\mathcal{Y}\mathcal{T}_H\mathbf{v}$ can be computed in $\mathcal{O}(mn \log{n})$ operations using FFTs and the related storage is of $\mathcal{O}(mn)$, where $m$ can be regarded as a fixed constant and $n$ is the effective dimensional parameter.

We first indicate that $\mathcal{P}$ has the following decomposition
\begin{align*}
	\mathcal{P}_H=  \sqrt{T_{(n,m)}[\vert g_{\theta}\vert^2]} &= \sqrt{ I_n\otimes (A_{(0)}^2 + A_{(1)}^2) + P_n \otimes 2(A_{(0)}A_{(1)}) }
\end{align*}
where
\begin{equation}\label{eqn:matrix_tri_p}
	P_n=\begin{bmatrix}
		0 & \frac{1}{2}  &  &  & \\
		\frac{1}{2}  & 0  & \frac{1}{2} & & \\
		&   \ddots & \ddots & \ddots &  \\
		&   & \ddots & \ddots & \frac{1}{2} \\
		&  &   & \frac{1}{2} & 0
	\end{bmatrix}
\end{equation}
is a tridiagonal Toeplitz matrix which has the eigendecomposition $P_n=S_n D_n S_n$ with $S_n=\sqrt{\frac{2}{n+1}}\Big[\sin{(\frac{ij \pi}{n+1})}\Big]^{n}_{i,j=1} \in \mathbb{R}^{n \times n}$ being the symmetric discrete sine transform matrix. Hence,
\begin{align*}
	\mathcal{P}_H&=  \sqrt{  I_n\otimes U_m( \Lambda^2_{(0)} + \Lambda^2_{(1)} )U_m^T + P_n \otimes (U_m 2\Lambda_{(0)}\Lambda_{(1)} U_m^T)} \\
	&=\sqrt{(S_n \otimes U_m)\Big(  I_n \otimes (\Lambda^2_{(0)} + \Lambda^2_{(1)} ) + D_n \otimes  2\Lambda_{(0)}\Lambda_{(1)}  \Big)(S_n \otimes U_m)^T}\\
	&=(S_n \otimes U_m)\sqrt{   I_n \otimes (\Lambda^2_{(0)} + \Lambda^2_{(1)} ) + D_n \otimes  2\Lambda_{(0)}\Lambda_{(1)} }(S_n \otimes U_m)^T.
\end{align*}

The product $\mathbf{z} = \mathcal{P}_H^{-1}\mathbf{v}$ can be computed via the following three-step procedures:

\begin{enumerate}[1.]
	\item $\textrm{Compute}~\widetilde{\mathbf{v}} =  (S_n \otimes U_m^T)\mathbf{v};$
	\item $\textrm{Compute}~\widetilde{\mathbf{z}} =\Big(\sqrt{   I_n \otimes (\Lambda^2_{(0)} + \Lambda^2_{(1)} ) + D_n \otimes  2\Lambda_{(0)}\Lambda_{(1)} } \Big)^{-1} \widetilde{\mathbf{v}};$
	\item $\textrm{Compute}~{\mathbf{z}} =  (S_n \otimes{ U_m})\widetilde{\mathbf{z}}$.
\end{enumerate}

When both $M_m$ and $K_m$ are diagonalizable by $U_m = S_m$, which holds true when the spatial discretization is a finite difference/element method with a uniform square grid and a constant diffusion coefficient function $a(x)$, the matrix-vector product $\mathbf{z} = \mathcal{P}_H^{-1}\mathbf{v}$ can be computed efficiently in $\mathcal{O}(nm (\log{m} + \log{n}))$ operations using fast sine transforms with a storage is of order $\mathcal{O}(nm)$.

{When the diffusion coefficients change corresponding to space, the matrix $K_m$ is no longer a Toeplitz matrix and is not diagonalizable by $S_m$. Considering the one-dimensional space case, let the spatial grid size $h=\frac{1}{m+1}$, we have the following coefficient matrix
	
	\begin{equation}\label{eqn:matrix_Km_va}
		K_m=\frac{1}{h^2}\begin{bmatrix}
			a_{\frac{1}{2}}+a_{\frac{3}{2}} & -a_{\frac{3}{2}} &  &  & \\
			-a_{\frac{3}{2}}& a_{\frac{3}{2}}+a_{\frac{5}{2}}  & -a_{\frac{5}{2}} & & \\
			&   \ddots & \ddots & \ddots &  \\
			&   & \ddots & \ddots &-a_{\frac{2m-1}{2}} \\
			&  &   & -a_{\frac{2m-1}{2}}& a_{\frac{2m-1}{2}}+a_{\frac{2m+1}{2}}
		\end{bmatrix},
	\end{equation}
	where  $a_j=x_{\textrm min}+jh$ for all $j\in [0,m]$. In this case, we can construct a preconditioner for $K_m$ by averaging is entries along the three \cred{diagonals} to obtain a tridiagonal symmetric Toeplitz matrix. Clearly, such a new preconditioner can be fast diagonalized by $S_m$. Thus, the overall preconditioner $ \mathcal{P}_H$ can be diagonalized as well.
}

\subsubsection{Eigenvalue analysis of $\mathcal{P}_H^{-1}\mathcal{Y}\mathcal{T}_H$}
{
	The following proposition guarantees the effectiveness of {\color{black} $\left(\sqrt{\mathcal{T}_H^{\top}\mathcal{T}_H}\right)^{-1}$} as a preconditioner for $\mathcal{Y}\mathcal{T}_H$, showing that the eigenvalues of {\color{black} $\left(\sqrt{\mathcal{T}_H^{\top}\mathcal{T}_H}\right)^{-1}\mathcal{Y} \mathcal{T}_H$} are clustered around 
	$\pm 1$ which ensures fast convergence of MINRES.

{\color{black}
\begin{proposition}\label{thm:absC}
	Let $\mathcal{T}_H $  and $\mathcal{Y} \mathcal{T}_H $ be defined in (\ref{theta}) and (\ref{eqn:main_system}), respectively. Then $$\left(\sqrt{\mathcal{T}_H^{\top}\mathcal{T}_H}\right)^{-1}\mathcal{Y} \mathcal{T}_H=\mathcal{Q}_0, $$ where $\mathcal{Q}_0$ is both symmetric and orthogonal matrix, and have only $-1$ and $1$ as eigenvalues.
\end{proposition}
\begin{proof}
	Since the eigendecomposition of the symmetric matrix $\mathcal{Y} \mathcal{T}_H$ is $\mathcal{Y} \mathcal{T}_H=E^{\top}\mathcal{Q} E$, we can have
	\begin{eqnarray*}
		\left(\sqrt{\mathcal{T}_H^{\top}\mathcal{T}_H}\right)^{-1}\mathcal{Y} \mathcal{T}_H&=&\left(\sqrt{(\mathcal{Y}\mathcal{T}_H)^{\top}\mathcal{Y}\mathcal{T}_H}\right)^{-1}\mathcal{Y} \mathcal{T}_H\\
		&=&\left(\sqrt{(\mathcal{Y}\mathcal{T}_H)^{2}}\right)^{-1}\mathcal{Y} \mathcal{T}_H\\
			&=&E^{\top}(\mathcal{Q} ^{2})^{-1/2}E E^{\top}\mathcal{Q} E\\
			&=&E^{\top}\hat{\mathcal{Q} } E\\
			&=&\mathcal{Q}_0 
		\end{eqnarray*}
	where $\hat{\mathcal{Q} }$ is a diagonal matrix whose entries are either $-1$ or $1$. Therefore, $\mathcal{Q}_0 $ is symmetric and orthogonal, and hence its eigenvalues are only $-1$ and $1$.
	\end{proof}}

Now, we turn our focus to the following lemma and proposition, which will be helpful to \cred{show} our main result.

{\color{black}
	\begin{lemma}\label{lemma:rankC_P}
	Let $\mathcal{T}_H, \mathcal{P}_H \in \mathbb{R}^{mn \times mn} $ be defined in (\ref{theta}) and (\ref{def:matrix_P}), respectively. Then,
	\[
	\textrm{rank}((\mathcal{T}_H^{\top}\mathcal{T}_H)^{K}-\mathcal{P}_H^{2K}) \leq Km,
	\]
	for some positive integer $K$ provided that $n>Km$.
\end{lemma}
\begin{proof}
	First, we observe that
	\begin{equation*}
		\mathcal{T}_H^{\top}\mathcal{T}_H-\mathcal{P}_H^2 =
		\begin{bmatrix}
			 &  &  & & & &  &  & \\
		 &  &  & & & &  &  & \\
			 &  &  & &  &&  &  & \\
			&  &  & &  &&  & & \\
			 &  &  & &  &&  &  & \\
			 &  & & & & &  &  & \\
			 & &  & & & &  &  &  -A_{(1)}^2\\
		\end{bmatrix}.
	\end{equation*}	
	Second, exploiting the simple structures of both $\mathcal{T}_H^{\top}\mathcal{T}_H$ and $\mathcal{P}_H^2$, we have by direct computations
	\begin{equation}\label{eqn:structure_CP}
		(\mathcal{T}_H^{\top}\mathcal{T}_H)^{ n_\alpha} (\mathcal{T}_H^{\top}\mathcal{T}_H-\mathcal{P}_H^2)\mathcal{P}_H^{2 n_\beta} =
		\begin{bmatrix}
			 &  &  & & & &  &  & \\
		 &  &  & & & &  &  & \\
			 &  &  & &  &&  &  & \\
			&  &  & &  &&  & & \\
			 &  &  & &  && * & \cdots & *\\
			 &  & & & & & \vdots &  & \vdots\\
			 & &  & & & & * & \cdots & *\\
		\end{bmatrix}
	\end{equation}
	for integer values $n_\alpha$ and $n_\beta$, where $*$ represents a nonzero entry. Namely, $(\mathcal{T}_H^{\top}\mathcal{T}_H)^{ n_\alpha} (\mathcal{T}_H^{\top}\mathcal{T}_H-\mathcal{P}_H^2)\mathcal{P}_H^{2 n_\beta} $ is a block matrix with one block in its Southeast corner, and the block is of size $(n_\alpha+1)m \times (n_\beta +1)m$. Thus,
	\begin{eqnarray*}
		(\mathcal{T}_H^{\top}\mathcal{T}_H)^{K}-\mathcal{P}_H^{2K} &=& \sum_{i=0}^{K-1} ((\mathcal{T}_H^{\top}\mathcal{T}_H)^{K-i}\mathcal{P}_H^{2i} - (\mathcal{T}_H^{\top}\mathcal{T}_H)^{K-i-1}\mathcal{P}_H^{2i+2})\\
		&=& \sum_{i=0}^{K-1} (\mathcal{T}_H^{\top}\mathcal{T}_H)^{K-i-1}(\mathcal{T}_H^{\top}\mathcal{T}_H - \mathcal{P}_H^2)\mathcal{P}_H^{2i}
	\end{eqnarray*}
	is also a block matrix with one block in its Southeast corner, which is of size $Km \times Km$ provided that $n>Km$. Hence, we have $\textrm{rank}((\mathcal{T}_H^{\top}\mathcal{T}_H)^{K}-\mathcal{P}_H^{2K}) \leq Km$.
\end{proof}
}
\begin{remark}\normalfont
	The computational lemma used in the proof of Lemma \ref{lemma:rankC_P} was first given in \cite[Lemma 3.11]{doi:10.1137/080720280}.
\end{remark}

{\color{black}
\begin{proposition}\label{Prepo:mainCP}
	Let $\mathcal{T}_H, \mathcal{P}_H \in \mathbb{R}^{mn \times mn} $ be defined in (\ref{theta}) and (\ref{def:matrix_P}), respectively. Then for any $\epsilon >0$ there exists an integer $K$ such that for all $n>Km$
	\[
	\left(\sqrt{\mathcal{T}_H^{\top}\mathcal{T}_H}\right)^{-1}-\mathcal{P}_H^{-1} = \mathcal{E} + \mathcal{R}_1,
	\]
	where $\|\mathcal{E}\|_2 \leq \epsilon$ and $\textrm{rank} (\mathcal{R}_1) \leq Km $. 
\end{proposition}
}

\begin{proof}
	Let $f(x)=x^{-1/2}$ and $F(x)$ be defined in equation (\ref{def:function_F}). By Theorem \ref{thm:Bernstein}, there exists a polynomial $p_K$ with degree less than or equal to $K$ such that
	{\color{black}
	\begin{align*}
		\|  \left(\sqrt{\mathcal{T}_H^{\top}\mathcal{T}_H}\right)^{-1} - p_K (\mathcal{T}_H^{\top}\mathcal{T}_H)\|_2 &= \vert (\mathcal{T}_H^{\top}\mathcal{T}_H)^{-\frac{1}{2}} - p_K (\mathcal{T}_H^{\top}\mathcal{T}_H)\|_2 \\
		 &= \max_{x \in \sigma(\mathcal{T}_H^{\top}\mathcal{T}_H)} \vert F(x) - p_K(x) \vert\\
		& \leq \| F -p_K(x)\|_{\infty} \\
		&\leq \frac{2M (\mathcal{X}_{\mathcal{T}_H^{\top}\mathcal{T}_H})}{\mathcal{X}_{\mathcal{T}_H^{\top}\mathcal{T}_H}-1}\cdot \frac{1}{\mathcal{X}_{\mathcal{T}_H^{\top}\mathcal{T}_H}^K},
	\end{align*}}
	\begin{align*}
		\| \mathcal{P}_H^{-1} - p_K (\mathcal{P}_H^2)\|_2 &= \| (\mathcal{P}_H^2)^{-\frac{1}{2}} - p_K (\mathcal{P}_H^2)\|_2 \\ &= \max_{x \in \sigma(\mathcal{P}_H^2)} \vert F(x) - p_K(x) \vert \\
		& \leq \| F -p_K(x)\|_{\infty} \\
		&\leq \frac{2M (\mathcal{X}_{\mathcal{P}_H^2})}{\mathcal{X}_{\mathcal{P}_H^2}-1}\cdot \frac{1}{\mathcal{X}_{\mathcal{P}_H^2}^K},
	\end{align*}
	where 
	\[ {\color{black}
	1 < \mathcal{X}_{\mathcal{T}_H^{\top}\mathcal{T}_H} < \frac{\sqrt{\kappa_{\mathcal{T}_H^{\top}\mathcal{T}_H}}+1}{\sqrt{\kappa_{\mathcal{T}_H^{\top}\mathcal{T}_H}}-1}},\qquad 1 < \mathcal{X}_{\mathcal{P}_H^2} < \frac{\sqrt{\kappa_{\mathcal{P}_H^2}}+1}{\sqrt{\kappa_{\mathcal{P}_H^2}}-1},
	\]
	and {\color{black}$\kappa_{\mathcal{T}_H^{\top}\mathcal{T}_H}$} and $\kappa_{\mathcal{P}_H^2}$ are the condition numbers of {\color{black}$\mathcal{T}_H^{\top}\mathcal{T}_H$} and $\mathcal{P}_H^2$, respectively. Thus, for any $\epsilon >0$ there exists an integer $K$ such that
	\[{\color{black}
	\| \left(\sqrt{\mathcal{T}_H^{\top}\mathcal{T}_H}\right)^{-1} - p_K (\mathcal{T}_H^{\top}\mathcal{T}_H)\|_2 \leq \epsilon }\quad \textrm{and} \quad \| \mathcal{P}_H^{-1} - p_K (\mathcal{P}_H^2)\|_2 \leq \epsilon.
	\]
	
	Also, we have
	{\color{black}
	\begin{eqnarray*}
		p_K(\mathcal{T}_H^{\top}\mathcal{T}_H) - p_K(\mathcal{P}_H^2) =\underbrace{\sum_{i=0}^{K} a_i ((\mathcal{T}_H^{\top}\mathcal{T}_H)^{i}-\mathcal{P}_H^{2i})}_{=:\mathcal{R}_1}.
	\end{eqnarray*}
}

	By Lemma \ref{lemma:rankC_P}, we know that $\mathcal{R}_1$ has the same sparsity structure as that of ${\color{black}(\mathcal{T}_H^{\top}\mathcal{T}_H)^{K}}-\mathcal{P}_H^{2K}$. Consequently, we deduce $\textrm{rank}(\mathcal{R}_1) \leq {\color{black}Km}$.

	We then obtain
	{\color{black}
	\begin{eqnarray*}
		&&\left(\sqrt{\mathcal{T}_H^{\top}\mathcal{T}_H}\right)^{-1} - \mathcal{P}_H^{-1}\\
		&&= \underbrace{(\mathcal{T}_H^{\top}\mathcal{T}_H)^{-\frac{1}{2}} - p_K(\mathcal{T}_H^{\top}\mathcal{T}_H) + p_K(\mathcal{P}_H^2) - \mathcal{P}_H^{-1}}_{\corr{=:\mathcal{E}}} + \underbrace{p_K(\mathcal{T}_H^{\top}\mathcal{T}_H) - p_K(\mathcal{P}_H^2)}_{\corr{=:}\mathcal{R}_1},
	\end{eqnarray*}}
	where $\|\mathcal{E}\|_2 \leq 2\epsilon$ and {\color{black} $\textrm{rank}(\mathcal{R}_1) \leq Km $}. Therefore, the proof is concluded.
	\end{proof}

\begin{theorem}\label{thm:absP_main}
	Let $\mathcal{Y}\mathcal{T}_H, \mathcal{P}_H \in \mathbb{R}^{mn \times mn} $ be defined in (\ref{eqn:main_system}) and (\ref{def:matrix_P}), respectively. Then for any $\epsilon >0$ there exists an integer $K$ such  that for all $n>{\color{black}Km}$
	\[
	\mathcal{P}_H^{-1}\mathcal{Y}\mathcal{T}_H = \mathcal{Q}_H + \mathcal{ {E}_{\textrm{H}}} + \mathcal{ {R}_{\textrm{H}}},
	\]
	where $ \mathcal{Q}_H$ is both symmetric and orthogonal, $\| \mathcal{ {E}_{\textrm{H}}} \|_2 \leq  \epsilon$, and $\textrm{rank} (\mathcal{ {R}_{\textrm{H}}}) \leq {\color{black}Km}.$
\end{theorem}

\begin{proof}
	By Proposition \ref{thm:absC}, we can write
	{\color{black}
	\begin{eqnarray*}
\left(\sqrt{\mathcal{T}_H^{\top}\mathcal{T}_H}\right)^{-1}\mathcal{Y} \mathcal{T}_H= \mathcal{Q}_0,
	\end{eqnarray*}
}

	By Proposition \ref{Prepo:mainCP}, we then have
	{\color{black}
	\begin{eqnarray*}
		\mathcal{P}_H^{-1}\mathcal{Y}\mathcal{T}_H & =&(\left(\sqrt{\mathcal{T}_H^{\top}\mathcal{T}_H}\right)^{-1}-\mathcal{E} -\mathcal{R}_1 )\mathcal{Y}\mathcal{T}_H\\
		&=& \mathcal{Q}_0 \underbrace{- \mathcal{E}\mathcal{Y}\mathcal{T}_H}_{=:\mathcal{ {E}}_H}   \underbrace{ - \mathcal{R}_1\mathcal{Y}\mathcal{T}_H }_{=:\mathcal{ {R}}_H},
	\end{eqnarray*}}
	where
	\begin{eqnarray*}
		\textrm{rank} ( \mathcal{ {R}}_H) {\color{black} = \textrm{rank}(\mathcal{R}_1\mathcal{Y}\mathcal{T}_H ) \leq Km}
	\end{eqnarray*}
	and
	\begin{eqnarray*}
		\|\mathcal{ {E}}_H\|_2 = \|-\mathcal{E}\mathcal{Y}\mathcal{T}_H\|_2 \leq \| \mathcal{Y}\mathcal{T}_H\|_2 \epsilon.
	\end{eqnarray*}
	Note that
	\begin{eqnarray*}
		\| \mathcal{Y}\mathcal{T}_H \|_2 = \|\mathcal{T}_H \|_2 \le  \| g_{\theta} \|_{\infty}.
	\end{eqnarray*}
	by using the general inequalities in \cite[Corollary 4.2]{Serra_Tilli_2002}, noticing that the Schatten $\infty$ norm in that paper is exactly the spectral norm $\|\cdot\|_2$. Furthermore, we have strict inequality, that is $\|YT\|_2=\|T\|_2< \| g_{\theta} \|_\infty$, whenever the infimum of $\vert g_{\theta} \vert$ is strictly bounded by $\| g_{\theta} \|_\infty$, as it often occurs, including in our case.
	
	Hence, $\| \mathcal{Y}\mathcal{T}_H \|_2$ is uniformly bounded with respect to $n$ and the proof is concluded.
\end{proof}

By \cite[Corollary 3]{BRANDTS20103100} and Theorem \ref{thm:main_hon-ter}, we know from Theorem \ref{thm:absP_main} that the preconditioned matrix sequence $\{ \mathcal{P}_H^{-1} \mathcal{Y} \mathcal{T}_H \}_n$ has clustered eigenvalues around $\pm1$, with a number of outliers independent of $n$. Hence, the convergence is independent of the time steps, and we can expect fast MINRES converges in exact arithmetic with $\mathcal{P}_H$ as a preconditioner.

\subsubsection{Eigenvalue analysis of $(\mathcal{P}_H^{-1}  \mathcal{T}_H)^T\mathcal{P}_H^{-1} \mathcal{T}_H$}

For a complete eigenvalue analysis, \cred{we examine} the preconditioned normal equation matrix $(\mathcal{P}_H^{-1}  \mathcal{T}_H)^T\mathcal{P}_H^{-1} \mathcal{T}_H$, even though we do not \cred{use} the conjugate gradient method on normal equations (CGNE) in this work.

As for implementations, the product $(\mathcal{P}_H^{-1}  \mathcal{T}_H)^T\mathcal{P}_H^{-1} \mathcal{T}_H\mathbf{v}$ for any vector $\mathbf{v}$ can be computed effectively using the similar arguments in Section \ref{sebsebsec:imple_MINRES}.

In the following result, we show that the eigenvalues of $(\mathcal{P}_H^{-1}  \mathcal{T}_H)^T\mathcal{P}_H^{-1} \mathcal{T}_H$ are clustered around $1$ with outliers independent of the time step $n$, which implies fast convergence of CGNE.

\begin{theorem}\label{thm:main_P_normal}
	Let $\mathcal{T}_H, \mathcal{P}_H \in \mathbb{R}^{mn \times mn}$ be defined in (\ref{theta}) and  (\ref{def:matrix_P}), respectively. Then,
	\begin{eqnarray*}
		(\mathcal{P}_H^{-1}  \mathcal{T}_H)^T\mathcal{P}_H^{-1} \mathcal{T}_H = I_{mn}+\mathcal{\widetilde {R}}_H,
	\end{eqnarray*}
	where $I_{mn}$ is the $mn$ by $mn$ identity matrix and $\textrm{rank} \corr{(\mathcal{\widetilde {R}}_{H})} \leq m$.
\end{theorem}

\begin{proof}
	Direct computations show that
	\begin{eqnarray*}\label{eqn:Psquare}
		\mathcal{P}_H^2-\mathcal{T}_H^T \mathcal{T}_H
		=
		\underbrace{\begin{bmatrix}
				O_m &   & & \\
				&   \ddots  &   &   \\
				&       &  O_m  &   \\
				&     &   & A^2_{(1)}  \\
		\end{bmatrix}}_{=:\mathcal{\widetilde {R}}_{0}},
	\end{eqnarray*} where $\textrm{rank}(\mathcal{\widetilde {R}}_{0}) \leq m$. The matrix equation further leads to
	\begin{eqnarray}\label{eqn:Psquare_inverse}
		I_{mn}-\mathcal{T}_H\mathcal{P}_H^{-2}\mathcal{T}_H^T=\underbrace{\mathcal{T}_H \mathcal{P}_H^{-2} \mathcal{\widetilde {R}}_{0} \mathcal{T}_H^{-1}}_{=:\mathcal{\widetilde {R}}_{1}},
	\end{eqnarray}
	where $\textrm{rank}(\mathcal{\widetilde {R}}_{1})=\textrm{rank}(\mathcal{\widetilde {R}}_{0}) \leq m$.
	
	We now consider the normal equation of $\mathcal{P}_H^{-1} \mathcal{Y} \mathcal{T}_H$, namely
	\begin{eqnarray*}
		(\mathcal{P}_H^{-1} \mathcal{Y} \mathcal{T}_H)^T \mathcal{P}_H^{-1} \mathcal{Y} \mathcal{T}_H=
		\mathcal{Y} \mathcal{T}_H \mathcal{P}_H^{-2}  \mathcal{T}_H^T \mathcal{Y}=\mathcal{Y} ( I_{mn}-\mathcal{\widetilde {R}}_{1}) \mathcal{Y}=  I_{mn} \underbrace{- \mathcal{Y}\mathcal{\widetilde {R}}_{1} \mathcal{Y}}_{=:\mathcal{\widetilde {R}}_H}
	\end{eqnarray*}
	where $\textrm{rank}(\mathcal{ \widetilde {R}}_H) =  \textrm{rank}(\mathcal{\widetilde {R}}_{1} )\leq  m $.
\end{proof}

Notice that every rank-one term in $\mathcal{\widetilde {R}}$ can perturb at the same time one eigenvalue of $1$. Therefore, by Theorem \ref{thm:main_P_normal}, there are at least \corr{$m(n-1)$} eigenvalues of the normal equation matrix are equal to $1$, which implies rapid convergence of CGNE.
\subsubsection{Modified Preconditioner}\label{subsec:MP}
\cred{
Following our spectral symbol matching strategy, the construction of the modified preconditioner $\mathcal{P}_{\theta}$ is motivated by the following approximation without requiring the fast diagonalizability of $M_{m}$ and $K_{m}$.

Since the original preconditioner $\mathcal{P}_H$ satisfy $\{\mathcal{P}_H^{-1}\mathcal{Y} \mathcal{T}_H\}_n$ are clustered around $\pm 1$ , where $	g_{\theta}(x)=A_{(0)}+A_{(1)}e^{\textbf{i}x}$ defined by (\ref{eqn:symbol_h}), we can approximate the spectral symbol $g_{\theta}$ 
\begin{eqnarray*}
    \sqrt{\vert g_{\theta} \vert^2} = \sqrt{A_{(0)}^2+A_{(1)}^2+2A_{(0)} A_{(1)} \cos(x)}
\end{eqnarray*}
where
\begin{eqnarray*}
&&A_{(0)}^2+A_{(1)}^2+2A_{(0)} A_{(1)} \cos(x) \\
&=& (2-2\cos(x))M_{m}^{2} + (\theta^2+(1-\theta)^2+2\theta(1-\theta)\cos(x))\tau^2 K_{m}^{2}\\ 
&&+ 2((2\theta-1)+(1-2\theta)\cos(x))\tau K_m M_m
\end{eqnarray*}
via 
the following matrix-valued function
\begin{eqnarray*}
    \widetilde{g_\theta} = \sqrt{2-2\cos(x)}M_m+ \tau \sqrt{\theta^2+(1-\theta)^2+2\theta(1-\theta)\cos(x)}  K_m
\end{eqnarray*}

Hence, we can regard $\mathcal{P}_{\theta}$ as satisfying $\{\mathcal{P}_{\theta}^{-1}\mathcal{Y} \mathcal{T}_H\}_n$ are clustered around $\pm 1$.

In matrix form, the corresponding preconditioner generated by $\widetilde{g_\theta}$ is precisely our proposed modified preconditioner $\mathcal{P}_{\theta}$ defined by (\ref{def:matrix_mod_P}).

The product $\mathcal{P}_{\theta}^{-1} \mathbf{v}$ for any vector $\mathbf{v}$ can be implemented by the following three-step procedures, which does not require the fast diagonalizability of $M_m$ and $K_m$. Notice that both $\mathcal{H}$ and $\mathcal{H}_{\theta}$ from (\ref{def:matrix_mod_P}) are Tau matrices. Hence,
\begin{align*}
    \mathcal{P}_{\theta} &= \mathcal{H} \otimes M_{m} + \mathcal{H}_{\theta} \otimes \tau K_{m}\\
                             &= (S_n \Lambda_{\mathcal{H}} S_n) \otimes M_{m} + (S_n \Lambda_{\mathcal{H}_{\theta}}S_n) \otimes \tau K_{m}\\
                         &= (S_n \otimes I_m) (\Lambda_{\mathcal{H}} \otimes M_{m} + \Lambda_{\mathcal{H}_{\theta}} \otimes K_m) (S_n \otimes I_m)^{T},
\end{align*}
where $\Lambda_{\mathcal{H}}$ and $\Lambda_{\mathcal{H}_{\theta}}$ are respectively the diagonal matrices containing the eigenvalues of $\mathcal{H}$ and $\mathcal{H}_{\theta}$.

Hence, the computations of $\mathbf{z} = \mathcal{P}_{\theta}^{-1} \mathbf{v}$ can be implemented via the following three steps:
\begin{enumerate}[1.]
	\item $\textrm{Compute}~\widetilde{\mathbf{v}} =  (S_n \otimes I_m)\mathbf{v};$
	\item $\textrm{Compute}~\widetilde{\mathbf{z}} =(\Lambda_{\mathcal{H}} \otimes M_{m} + \Lambda_{\mathcal{H}_{\theta}} \otimes K_m )^{-1} \widetilde{\mathbf{v}};$
	\item $\textrm{Compute}~{\mathbf{z}} =  (S_n \otimes{ I_m})\widetilde{\mathbf{z}}$.
\end{enumerate}

Both Steps 1 \& 3 can be computed efficiently via fast sine transforms in $\mathcal{O}(mn\log{n})$ operations. As for Step 2, the shifted Laplacian systems can be efficiently solved for example using the multigrid methods \cite{doi:10.1137/15M102085X}. For more details regarding such efficient implementation, we refer to \cite[Section 3.1]{doi:10.1137/16M1062016}.

}

	\section{Extension to multistep methods}\label{sec:extension}

Our proposed preconditioning method can be easily extended for high-order in time discretization schemes. Consider again (\ref{eqn:heat}), using an $l$-order backward difference scheme for time, we have the following all-at-once system $\widehat{\mathcal{T}}\mathbf{u}=\widehat{\mathbf{f}}$ to solve for $\mathbf{u}$, where
\begin{equation}\label{eqn:matrix_gen}
	\widehat{\mathcal{T}}=\begin{bmatrix}
		A_{(0)} &   &  & &  & \\
		A_{(1)}  & \ddots   & & &  &\\
		\vdots & \ddots   & \ddots& &  & \\
		A_{(l)}  & \ddots & \ddots & \ddots& &  \\
		&   \ddots& \ddots &  \ddots & \ddots & \\
		&   & \ddots &  \ddots & \ddots & \ddots \\
		& &  & A_{(l)}   & \cdots & A_{(1)} & A_{(0)}
	\end{bmatrix}
\end{equation}
is {a $mn$ by $mn$ multiple block Toeplitz matrix, which is} generated by the matrix-valued function
\begin{equation}\label{eqn:symbol_h_gen}
	\widehat{g}(x) = A_{(0)} + A_{(1)}e^{\textbf{i}x} +  \cdots  + A_{(l)}e^{\textbf{i}lx},
\end{equation}
and $\widehat{\mathbf{f}}$ is an appropriate vector resulted from the adopted discretization schemes. Note that the matrices $A_{(k)}, k=1,\dots, l,$ are composed of the mass matrix $M_m$ and the stiffness matrix $K_m$.

Our approach is to first permute (\ref{eqn:matrix_gen}) and obtain a symmetric (indefinite) matrix
\begin{equation}\label{eqn:main_system_gem_sym}
	\mathcal{Y} \widehat{\mathcal{T}}
	=
	\begin{bmatrix}
		&   &  & A_{(l)}& \cdots & A_{(1)} & A_{(0)} \\
		&   & \iddots& \iddots&  \iddots& \iddots\\
		& \iddots   & \iddots& \iddots &  \iddots & \\
		A_{(l)}  & \iddots & \iddots & \iddots& &  \\
		\vdots &   \iddots& \iddots &   &  & \\
		A_{(1)}   & \iddots   &  &   &  &  \\
		A_{(0)}  & &  &    &  & &
	\end{bmatrix},
\end{equation} and then construct a corresponding sine transform based preconditioner $\widehat{\mathcal{P}}$ for it. We first compute the following matrix-valued function from (\ref{eqn:symbol_h_gen})
\begin{eqnarray}\label{eqn:symbol_h_gen_poly}
	\vert \widehat{g}(x) \vert^2  &=&  \underbrace{A_{(0)}^2 + A_{(1)}^2 + \cdots + A_{(l)}^2 }_{=:q_0(A)}\\\nonumber
	&& + \underbrace{( A_{(0)}A_{(1)} + A_{(1)}A_{(2)} + \cdots + A_{(l-1)}A_{(l)} ) }_{=:q_1(A)}2\cos{(x)}\\\nonumber
	&& + \cdots + \underbrace{(A_{(0)}A_{(l)})}_{=:q_l(A)} 2 \cos{(lx)},
\end{eqnarray}
where \cred{$q_k, k=0,\dots, l,$} are all matrix polynomials composed of \cred{$A_{(k)}, k=0,\dots, l$}. Now, by expressing $\cos{(kx)},k=1,\dots, l,$ in terms of $\cos{(x)}$, we can further rewrite $\vert \widehat{g} \vert^2 $ as
\begin{eqnarray}\nonumber
	\vert \widehat{g}(x) \vert^2  &=&  \overline{q}_0(A) + \overline{q}_1(A)\cos{(x)} + \cdots + \overline{q}_l(A) \cos^l{(x)},
\end{eqnarray}
where \cred{$\overline{q}_k, k=0,\dots, l,$} are certain matrix polynomials composed of \cred{$A_{(k)}, k=0,\dots, l.$}

	We propose the following SPD preconditioner, which approximates the symbol $\vert \widehat{g} \vert$ in the eigenvalue sense, for $\mathcal{Y} \widehat{\mathcal{T}}$:
\begin{align}\label{def:matrix_P_gen}
	\cred{\widehat{\mathcal{P}}}= \sqrt{ I_n\otimes \overline{q}_0(A) + P_n \otimes \overline{q}_1(A)  + \cdots + P_n^l \otimes \overline{q}_l(A) },
\end{align}
where $P_n$ is the same matrix given by (\ref{eqn:matrix_tri_p}).

Note that
\begin{align*}
	\cred{\widehat{\mathcal{P}}}&=  \sqrt{ I_n\otimes U_m \overline{q}_0(\Lambda) U_m^T + P_n \otimes U_m \overline{q}_1(\Lambda) U_m^T  + \cdots + P_n^l \otimes U_m \overline{q}_l(\Lambda) U_m^T } \\
	&=\sqrt{(S_n \otimes U_m)\big(  I_n \otimes \overline{q}_0(\Lambda) + D_n \otimes \overline{q}_1(\Lambda) + \cdots + D_n^l \otimes \overline{q}_l(\Lambda)   \big)(S_n \otimes U_m)^T}\\
	&=(S_n \otimes U_m)\sqrt{   I_n \otimes \overline{q}_0(\Lambda) + D_n \otimes \overline{q}_1(\Lambda) + \cdots + D_n^l \otimes \overline{q}_l(\Lambda)   }(S_n \otimes U_m)^T.
\end{align*}

Hence, the product $\cred{{\widehat{\mathcal{P}}}^{-1}}\mathbf{v}$ for any vector $\mathbf{v}$ can be computed effectively using the similar three-step procedures in Section \ref{sebsebsec:imple_MINRES}.

	\subsection{Eigenvalue analysis of $\widehat{\mathcal{P}}^{-1}\mathcal{Y}\widehat{\mathcal{T}}$}

As before, we first introduce {\color{black} matrix  $\left(\sqrt{ \widehat{\mathcal{T}}^{\top} \widehat{\mathcal{T}}}\right)^{-1}$} as an auxiliary for showing the eigenvalues of  $\widehat{\mathcal{P}} ^{-1}\mathcal{Y}  \widehat{\mathcal{T}}$  are clustered.
The following proposition guarantees the effectiveness of {\color{black} $\sqrt{ \widehat{\mathcal{T}}^{\top} \widehat{\mathcal{T}}}$} as a preconditioner for $\mathcal{Y}\widehat{\mathcal{T}}$, showing that the eigenvalues of {\color{black}$\{\left(\sqrt{ \widehat{\mathcal{T}}^{\top} \widehat{\mathcal{T}}}\right)^{-1}\mathcal{Y}  \widehat{\mathcal{T}}\}_n$} are clustered around $\pm 1$, which leads to fast convergence of MINRES. 
{\color{black}
\begin{proposition}\label{thm:absC_gen}
	Let $\mathcal{Y}\widehat{\mathcal{T}}, \widehat{\mathcal{T}} \in \mathbb{R}^{mn \times mn}$ be defined in (\ref{eqn:main_system_gem_sym}) and  (\ref{eqn:matrix_gen}), respectively. Then,
$
	\left(\sqrt{ \widehat{\mathcal{T}}^{\top} \widehat{\mathcal{T}}}\right)^{-1}\mathcal{Y} \widehat{\mathcal{T}}$ is symmetric and orthogonal, and eigenvalues are only $1$ and $-1$.
\end{proposition}
\begin{proof}
		Since the eigendecomposition of the symmetric matrix $\mathcal{Y} \widehat{\mathcal{T}}$ is $\mathcal{Y} \widehat{\mathcal{T}}=E^{\top}\mathcal{Q}_0 E$, we can have
	\begin{eqnarray*}
		\left(\sqrt{\widehat{\mathcal{T}}^{\top}\widehat{\mathcal{T}}}\right)^{-1}\mathcal{Y} \widehat{\mathcal{T}}&=&\left(\sqrt{(\mathcal{Y}\widehat{\mathcal{T}})^{\top}\mathcal{Y}\widehat{\mathcal{T}}}\right)^{-1}\mathcal{Y} \widehat{\mathcal{T}}\\
		&=&\left(\sqrt{(\mathcal{Y}\widehat{\mathcal{T}})^{2}}\right)^{-1}\mathcal{Y} \widehat{\mathcal{T}}\\
		&=&E^{\top}(\mathcal{Q}_0 ^{2})^{-1/2}E E^{\top}\mathcal{Q}_0 E\\
		&=&E^{\top}\widehat{\mathcal{Q}_0 } E
	\end{eqnarray*}
	where $\widehat{\mathcal{Q}_0 }$ is a diagonal matrix whose entries are either $-1$ or $1$. Therefore, $		\left(\sqrt{ \widehat{\mathcal{T}}^{\top} \widehat{\mathcal{T}}}\right)^{-1}\mathcal{Y} \widehat{\mathcal{T}}$ is symmetric and orthogonal, and hence its eigenvalues are only $-1$ and $1$.
	\end{proof}}

In addition, we require the following lemma and proposition for showing our main result.
	
\begin{lemma}\label{lemma:rankC_P_gen}
	Let ${\color{black}\widehat{\mathcal{T}}}, \widehat{\mathcal{P}} \in \mathbb{R}^{mn \times mn} $ be defined in (\ref{eqn:matrix_gen}) and (\ref{def:matrix_P_gen}), respectively. Then,
	\[
	\textrm{rank}({\color{black}{( \widehat{\mathcal{T}}^{\top} \widehat{\mathcal{T}})}^{K}}-\widehat{\mathcal{P}}^{2K}) \leq \corr{2Klm}
	\]
	for some positive integer $K$ provided that $n>\corr{2Klm}$.
\end{lemma}

\begin{proof}
	First, we observe from (\ref{eqn:symbol_h_gen_poly}) that
	\begin{eqnarray*}
		{T}_{(n,m)}[\vert \widehat{g} \vert ^2]&=&{T}_{(n,m)}[q_0(A) +q_1(A)2\cos{(x)}+ \cdots + q_l(A)2\cos{(lx)}]\\
		&=&\begin{bmatrix}
			q_0(A)& q_1(A)  &\cdots & q_l(A) &  & & &  &\\
			q_1(A)  & \ddots   & \ddots& \ddots & \ddots & & &  & \\
			\vdots & \ddots   & & &   & \ddots &   & &  \\
			q_l(A)  & \ddots &  & & & & \ddots   &  \\
			&   \ddots&  &   &  & & & \ddots &  \\
			&   & \ddots  &   &   &   &   & \ddots & q_l(A)  \\
			&   &   &  \ddots &   &   &   &  \ddots & \vdots  \\
			&  & &  &  \ddots & \ddots & \ddots &\ddots &q_1(A) \\
			&  &  & &  & q_l(A)   & \cdots & q_1(A)& q_0(A)
		\end{bmatrix}.
	\end{eqnarray*}
	Note that we also have
\begin{eqnarray*}
	&&{T}_{(n,m)}[\vert\widehat{g}\vert^2] \\
	&=&  {T}_{(n,m)}[ \overline{q}_0(A) + \overline{q}_1(A)\cos{(x)} + \overline{q}_2(A)\cos^2{(x)} + \cdots + \overline{q}_l(A) \cos^l{(x)}]\\
	&=&{T}_{(n,m)}[\overline{q}_0(A)] + {T}_{(n,m)}[\overline{q}_1(A)\cos{(x)}] + {T}_{(n,m)}[\overline{q}_2(A)\cos^2{(x)}] \\
	&&+ \cdots + {T}_{(n,m)}[ \overline{q}_l(A) \cos^l{(x)}]\\
	&=& I_n \otimes \overline{q}_0(A) +T_{(n,1)}[\cos{(x)}]\otimes \overline{q}_1(A) + T_{(n,1)}[\cos^2{(x)}]\otimes \overline{q}_2(A)\\
	&&+ \cdots + T_{(n,1)}[\cos^l{(x)}] \otimes \overline{q}_l(A)\\
	&=& I_n \otimes \overline{q}_0(A) + T_{(n,1)}[\cos{(x)}]\otimes \overline{q}_1(A) + ( T_{(n,1)}[\cos{(x)}]^2+ {R}_{(2)} )\otimes \overline{q}_2(A)\\
	&& + \cdots + (T_{(n,1)}[\cos{(x)}]^l + {R}_{(l)} ) \otimes \overline{q}_l(A),
\end{eqnarray*}
where ${R}_{(k)},k=2,\dots,l,$ are low rank matrices in the sense that $\textrm{rank}({R}_{(k)}) \leq 2(k-1)$.

Recalling from (\ref{def:matrix_P_gen}) that $\mathcal{\widehat{P}}^2=I_n\otimes \overline{q}_0(A) + P_n \otimes \overline{q}_1(A)  + \cdots + P_n^l \otimes \overline{q}_l(A)$ and also note that $P_n=T_{(n,1)}[\cos{(x)}]$, we have
\begin{eqnarray*}
	{T}_{(n,m)}[\vert\widehat{g}\vert^2] 
	&=& I_n \otimes \overline{q}_0(A) + P_n \otimes \overline{q}_1(A) + (\corr{P_n^2} + {R}_{(2)}) \otimes \overline{q}_2(A)\\
	&& + \cdots + (P_n^l + {R}_{(l)} ) \otimes \overline{q}_l(A)\\
	&=& \underbrace{I_n \otimes \overline{q}_0(A) + P_n \otimes \overline{q}_1(A) + \corr{P_n^2} \otimes \overline{q}_2(A) + \cdots + P_n^l \otimes \overline{q}_l(A)}_{=\mathcal{\widehat{P}}^2}\\
	&& +  \underbrace{{R}_{(2)} \otimes \overline{q}_2(A) + \cdots + {R}_{(l)} \otimes \overline{q}_l(A)}_{\widehat{\mathcal{R}}_1}.
\end{eqnarray*}

One can further show by direct computations that $\widehat{\mathcal{R}}_1$ has the following structure
\[
\corr{
\widehat{\mathcal{R}}_1
=
\begin{bmatrix}
	* & \cdots & * & & & &  &  & \\
	\vdots &  & \vdots & & & &  &  & \\
	* & \cdots & * & &  &&  &  & \\
	&  &  & &  &&  & & \\
	&  &  & &  && * & \cdots & *\\
	&  &  & & & & \vdots &  & \vdots\\
	&  &  & & & & * & \cdots & *\\
\end{bmatrix}
}
\] where $*$ represents a nonzero entry, since such a block structure is determined by the last term ${R}_{(l)} \otimes \overline{q}_l(A)$. Namely, it is a block matrix with in the \corr{Southeast and Northwest corners} and the block is of size {\color{black}$lm \times lm$}. 

We can now examine the structure of the matrix ${\color{black} \widehat{\mathcal{T}}^{\top} \widehat{\mathcal{T}}}-\widehat{\mathcal{P}}^2$. By (\ref{eqn:matrix_gen}), we have
\corr{
\begin{eqnarray*}
\widehat{\mathcal{T}}^{\top} \widehat{\mathcal{T}} - \widehat{\mathcal{P}}^2 &=& \widehat{\mathcal{T}}^{\top} \widehat{\mathcal{T}} - {T}_{(n,m)}[\vert \widehat{g}\vert^2] +{T}_{(n,m)}[\vert \widehat{g}\vert^2] -\widehat{\mathcal{P}}^2 \\
&=&
\begin{bmatrix}
 & & & & & & & &\\
&  & & & & & & &\\
 &  &  & &  && & &\\
 &  &  & &  &&  & & \\
& & & &  && * & \cdots & *\\
& & & & & & \vdots &  & \vdots\\
& & & & & & * & \cdots & *\\
\end{bmatrix}
+
\widehat{\mathcal{R}}_1,
\end{eqnarray*}
where the first matrix is a block matrix with a $lm \times lm$ block in the corner.} Thus, ${\color{black} \widehat{\mathcal{T}}^{\top} \widehat{\mathcal{T}}}-\widehat{\mathcal{P}}^2$ is a block matrix with {\color{black} a $lm \times lm$ block} in the \corr{two corners}. We have by direct computations{\color{black}
\[
(\widehat{\mathcal{T}}^{\top} \widehat{\mathcal{T}})^{ n_\alpha} (\widehat{\mathcal{T}}^{\top} \widehat{\mathcal{T}}-\widehat{\mathcal{P}}^2)\widehat{\mathcal{P}}^{2 n_\beta} =
\corr{
\begin{bmatrix}
* & \cdots & *& & & &  & & \\
\vdots &  & \vdots & & & & &  &\\
* & \cdots & * & &  &&  & &  \\
 &  &  & &  &&  & & \\
 & & & &  && * & \cdots & *\\
 &  & & & & & \vdots &  & \vdots\\
 & & & & & & * & \cdots & *\\
\end{bmatrix}.
}
\]}

Similar to (\ref{eqn:structure_CP}), by exploiting the simple structure of {\color{black}$\widehat{\mathcal{T}}^{\top} \widehat{\mathcal{T}} - \widehat{\mathcal{P}}^2$}, we can show that for given positive integers $n_\alpha$ and $n_\beta$ the matrix {\color{black} $(\widehat{\mathcal{T}}^{\top} \widehat{\mathcal{T}})^{ n_\alpha} (	\widehat{\mathcal{T}}^{\top} \widehat{\mathcal{T}} -\widehat{\mathcal{P}}^2)\widehat{\mathcal{P}}^{2 n_\beta} $} is a block matrix with \corr{a block in its Southeast and Northwest corners} and the block is of size $(n_\alpha+1)lm \times (n_\beta +1)lm$. Thus,
\begin{eqnarray*}
		(\widehat{\mathcal{T}}^{\top} \widehat{\mathcal{T}})^{K}-\widehat{\mathcal{P}}^{2K}
	&=& \sum_{i=0}^{K-1} (\widehat{\mathcal{T}}^{\top} \widehat{\mathcal{T}})^{K-i-1}(\widehat{\mathcal{T}}^{\top} \widehat{\mathcal{T}} - \widehat{\mathcal{P}}^2)\widehat{\mathcal{P}}^{2i}
\end{eqnarray*}
is also a block matrix with \corr{a block in the two corners and the size is $Klm \times Klm$}, provided that $n>\corr{2Klm}$. Hence, we have \corr{$\textrm{rank}((\widehat{\mathcal{T}}^{\top} \widehat{\mathcal{T}})^K-\widehat{\mathcal{P}}^{2K}) \leq 2Klm$.}
\end{proof}

	The following results, i.e., Proposition \ref{Prepo:mainCP_gen} and Theorem \ref{thm:absP_main_gen}, follow using similar arguments for showing \cred{Proposition} \ref{Prepo:mainCP} and Theorem \ref{thm:absP_main}. Hence, the proofs are left to the readers.

\begin{proposition}\label{Prepo:mainCP_gen}
	Let ${\color{black}\widehat{\mathcal{T}}}, \widehat{\mathcal{P}} \in \mathbb{R}^{mn \times mn} $ be defined in (\ref{eqn:matrix_gen}) and (\ref{def:matrix_P_gen}), respectively. Then for any $\epsilon >0$ there exists an integer $K$ such that for all $n>{\color{black}2Klm}$
	\[{\color{black}
	(\widehat{\mathcal{T}}^{\top} \widehat{\mathcal{T}})^{-1/2}}-\widehat{\mathcal{P}}^{-1} = \widehat{\mathcal{E}}_0 + \widehat{\mathcal{R}}_2,
	\]
	where $\|\widehat{\mathcal{E}}_0\|_2 \leq \epsilon$ and $\textrm{rank} ( \widehat{\mathcal{R}}_2) \leq {\color{black}2Klm}$. 
\end{proposition}

\begin{theorem}\label{thm:absP_main_gen}
	Let $\mathcal{Y}\mathcal{\widehat{T}}, \mathcal{\widehat{P}} \in \mathbb{R}^{mn \times mn} $ be defined in (\ref{eqn:main_system_gem_sym}) and (\ref{def:matrix_P_gen}), respectively. Then for any $\epsilon >0$ there exists an integer $K$ such  that for all $n>{\color{black}2Klm}$
	\[
	\mathcal{\widehat{P}}^{-1}\mathcal{Y}\mathcal{\widehat{T}} = \mathcal{\widehat {Q}} + \mathcal{\widehat {E}} + \mathcal{\widehat {R}},
	\]
	where $ \mathcal{\widehat {Q}} $ is both symmetric and orthogonal, $\|\mathcal{\widehat {E}}\|_2 \leq  \epsilon$, and $\textrm{rank} (\mathcal{\widehat {R}}) \leq {\color{black}2Klm}.$
\end{theorem}

Since both $\mathcal{\widehat {E}}$ and $\mathcal{\widehat {R}}$ in Theorem \ref{thm:absP_main_gen} are symmetric, by \cite[Corollary 3]{BRANDTS20103100} and Theorem \ref{thm:main_hon-ter}, we know that the preconditioned matrix sequence $\{ \mathcal{\widehat{P}}^{-1} \mathcal{Y} \mathcal{\widehat{T}} \}_n$ has clustered eigenvalues around $\pm1$ with number of outliers independent of $n$. Hence, the convergence is independent of the time steps, and we can then expect that MINRES converges rapidly in exact arithmetic with $\mathcal{\widehat{P}}$ as a preconditioner.

	\section{Numerical examples}\label{sec:numerical}

In this section, we demonstrate the effectiveness of our proposed preconditioner. All numerical experiments are carried out using MATLAB on \cred{a Dell R640 server equipped with dual Xeon Gold 6246R 16-Cores 3.4GHz CPUs, 512GB RAM running Ubuntu 20.04 LTS}. The CPU time, which is denoted  by ``CPU'' in the tables below, is estimated in seconds using the MATLAB built-in \cred{functions} \textbf{tic/toc}. All Krylov subspace solvers are implemented using the build-in functions on MATLAB, and ``Iter'' refers to the iteration number required for convergence. Furthermore, we choose a zero initial guess and a stopping tolerance of $10^{-6}$ based on the reduction in relative residual norms. Throughout all examples, we consider finite difference methods with uniform spatial grids, which results in $M_m$ being the identity matrix and $K_m $ being diagonalized by the discrete sine transform. Note that in the examples ``DoF'' denotes the degree of freedom, and $\mathcal{C}_H$ is the existing absolute value block circulant preconditioner proposed in \cite{doi:10.1137/16M1062016}.

\begin{example}\normalfont \cite{doi:10.1137/16M1062016}\label{example_3}
	The \cred{first} example is a two-dimensional problem of solving (\ref{eqn:heat}) with $\Omega=(0,1) \times (0,1)$, $u_0(x,y)=x(x-1)y(y-1)$, $a(x,y)=10^{-5}$, $g=0$, $f=0${, and $T=1$}.
	
	For this example, we adopt the backward Euler method and \cred{the Crank-Nicolson method} in time for solving the equation and report the convergence results in Tables \ref{tab:example_3_MINRES_theta_one} \cred{and \ref{tab:example_3_crank}}, respectively. While $\mathcal{C}_H$ increases quickly in iteration number with the parameters for this ill-conditioned example, \cred{our proposed preconditioners $\mathcal{P}_H$ and $\mathcal{P}_{\theta}$ significantly outperform $\mathcal{C}_H$ in terms of both iteration numbers and CPU times. In addition, we observe that $\mathcal{P}_{\theta}$ is only slightly inferior in precondition effect when compared with $\mathcal{P}_H$. In fact, in most cases, their induced iteration numbers are identical. The use of $\mathcal{P}_{\theta}$ as a preconditioner appears to be excellent, considering that it is an approximation to $\mathcal{P}_H$.}

	\begin{table}[h]
		\begin{center}
			\begin{minipage}{\textwidth}
				\caption{Convergence results with MINRES for Example \ref{example_3} when $\theta=1$ (the backward Euler method)}\label{tab:example_3_MINRES_theta_one}
				\color{black}\begin{tabular*}{\textwidth}{@{\extracolsep{\fill}}lccccccccc@{\extracolsep{\fill}}}
					\toprule%
					\multicolumn{3}{@{}c@{}}{} & \multicolumn{2}{@{}c@{}}{$\mathcal{C}_H$} & \multicolumn{2}{@{}c@{}}{$\mathcal{P}_H$} & \multicolumn{2}{@{}c@{}}{$\mathcal{P}_{\theta}$} \\ \cmidrule{4-5}\cmidrule{6-7}\cmidrule{8-9}%
					$n$     & $m+1$  & DoF       & Iter             & CPU               & Iter             & CPU & Iter             & CPU    \\
					\midrule
					\multirow{4}{*}{$2^5$} 	&$	2^5	$&	30752	&	34	&	0.24	&	11	&	0.096	& 11 & 0.071\\
					&$	2^6	$&	127008	&	48	&	0.65	&	11	&	0.17 &  11 & 0.18	\\
					&$	2^7	$&	516128	&	59	&	2.18	&	11	&	0.57	& 11 & 0.56\\
					&$	2^8	$&	2080800	&	82	&	19.56	&	11	&	3.27	& 11 & 3.27 \\
					\midrule	\multirow{4}{*}{$2^6$} 	&$	2^5	$&	61504	&	34	&	0.30	&	11	&	0.12	& 11 & 0.11\\
					&$	2^6	$&	254016	&	48	&	1.04	&	11	&	0.26	& 11 & 0.27\\
					&$	2^7	$&	1032256	&	72	&	6.81	&	11	&	1.35	& 13 & 1.48\\
					&$	2^8	$&	4161600	&	82	&	40.82	&	11	&	7.28	& 13 & 8.37\\
					\midrule	\multirow{4}{*}{$2^7$} 	&$	2^5	$&	123008	&	34	&	0.49	&	13	&	0.21	& 13 & 0.22 \\
					&$	2^6	$&	508032	&	48	&	1.92	&	13	&	0.62	& 13 & 0.61\\
					&$	2^7	$&	2064512	&	72	&	15.04	&	13	&	3.68	& 13 & 3.73\\
					&$	2^8	$&	8323200	&	79	&	81.81	&	13	&	17.19	& 13 & 17.32\\
					\midrule	\multirow{4}{*}{$2^8$} 	&$	2^5	$&	246016	&	34	&	0.78	&	13	&	0.35	& 15 & 0.38\\
					&$	2^6	$&	1016064	&	48	&	4.62	&	13	&	1.61	& 15 & 1.78\\
					&$	2^7	$&	4129024	&	71	&	34.91	&	13	&	8.28	& 15 & 9.34\\
					&$	2^8	$&	16646400	&	79	&	157.74	&	14	&	35.95	& 15 & 38.32\\
					\botrule
				\end{tabular*}
			\end{minipage}
		\end{center}
	\end{table}

    \begin{table}[h]
		\begin{center}
			\begin{minipage}{\textwidth}
				\caption{\cred{Convergence results with MINRES for Example \ref{example_3} when $\theta=0.5$ (the Crank-Nicolson method)}}\label{tab:example_3_crank}
				\color{black}\begin{tabular*}{\textwidth}{@{\extracolsep{\fill}}lccccccccc@{\extracolsep{\fill}}}
					\toprule%
					\multicolumn{3}{@{}c@{}}{} & \multicolumn{2}{@{}c@{}}{$\mathcal{C}_H$} & \multicolumn{2}{@{}c@{}}{$\mathcal{P}_H$} & \multicolumn{2}{@{}c@{}}{$\mathcal{P}_{\theta}$} \\ \cmidrule{4-5}\cmidrule{6-7}\cmidrule{8-9}%
					$n$     & $m+1$  & DoF       & Iter             & CPU               & Iter             & CPU & Iter             & CPU    \\
					\midrule
					\multirow{4}{*}{$2^5$} 	&$	2^5	$&	30752	&	33	&	0.45	&	11	&	0.098	& 11 & 0.082\\
					&$	2^6	$&	127008	&	48	&	0.67	&	11	&	0.18 &  11 & 0.20	\\
					&$	2^7	$&	516128	&	59	&	2.42	&	11	&	0.63	& 11 & 0.63\\
					&$	2^8	$&	2080800	&	82	&	18.83	&	11	&	3.44	& 11 & 3.49 \\
					\midrule	\multirow{4}{*}{$2^6$} 	&$	2^5	$&	61504	&	34	&	0.38	&	11	&	0.15	& 11 & 0.14\\
					&$	2^6	$&	254016	&	48	&	1.10	&	11	&	0.30	& 13 & 0.36\\
					&$	2^7	$&	1032256	&	73	&	7.28	&	11	&	1.44	& 13 & 1.51\\
					&$	2^8	$&	4161600	&	83	&	44.01	&	11	&	7.64	& 13 & 8.74\\
					\midrule	\multirow{4}{*}{$2^7$} 	&$	2^5	$&	123008	&	34	&	0.68	&	13	&	0.23	& 13 & 0.24 \\
					&$	2^6	$&	508032	&	48	&	2.20	&	13	&	0.71	& 13 & 0.67\\
					&$	2^7	$&	2064512	&	72	&	16.06	&	13	&	3.98	& 13 & 3.97\\
					&$	2^8	$&	8323200	&	80	&	87.92	&	13	&	18.48	& 13 & 19.00\\
					\midrule	\multirow{4}{*}{$2^8$} 	&$	2^5	$&	246016	&	34	&	0.83	&	13	&	0.39	& 15 & 0.42\\
					&$	2^6	$&	1016064	&	48	&	5.02	&	13	&	1.73	& 15 & 1.97\\
					&$	2^7	$&	4129024	&	72	&	37.34	&	13	&	9.14	& 15 & 10.34\\
					&$	2^8	$&	16646400	&	79	&	165.53	&	14	&	37.91	& 15 & 40.49\\
					\botrule
				\end{tabular*}
			\end{minipage}
		\end{center}
	\end{table}
\end{example}

    \begin{example} \label{example_2D_Lin}
        \cred{This example is a two-dimensional problem of solving (\ref{eqn:heat}) with $\Omega=(0,1)^2$, $u_0(x,y)=x(1-x)y(1-y)$, $a(x,y)=10^{-5}\sin(\pi x y)$, $g=0, T=1$ and 
        \begin{multline*}
        f(x,y,t) = e^{-t} x (1-x) [2\times 10^{-5}\sin(\pi x y ) - y (1-y) - \pi \times 10^{-5} \cos(\pi x y) x (1-2y)]\\
                    +e^{-t} y (1-y) [2\times 10^{-5} \sin(\pi x y) - \pi \times 10^{-5} \cos(\pi x y) y (1-2x)].
        \end{multline*}
        This example has the closed-form analytical solution as follows
        \begin{eqnarray*}
            u(x,y,t) = e^{-t} x (1-x) y (1-y).
        \end{eqnarray*}
        Since the exact solution of this example is known, we can define the error estimate as follows
        \begin{eqnarray*}
            e_{n,m} = \| \mathbf{u} - \mathbf{u^*} \|_{\infty}
        \end{eqnarray*}
        where $\mathbf{u}$ denotes the iterative solution of the linear system in (\ref{theta}), and $\mathbf{u^*}$ denotes the exact solution of the heat equation on the mesh. 
        
        Tables \ref{tab:case_2D_Lin_backward} and \ref{tab:case_2D_Lin_Crank} show the convergence results of MINRES using the backward Euler method and the Crank-Nicolson method in time, respectively. Table \ref{tab:case_2D_Lin_error} shows the corresponding error results of MINRES using different preconditioners. Again, similar to the last example, we observed that our proposed preconditioners $\mathcal{P}_H$ and $\mathcal{P}_\theta$ considerably surpass $\mathcal{C}_H$ in this variable-coefficient case.}

        \begin{table}[h]
			\begin{center}
				\begin{minipage}{\textwidth}
					\caption{\cred{Convergence results with MINRES for Example \ref{example_2D_Lin} when $\theta=1$ (the backward Euler method)}}\label{tab:case_2D_Lin_backward}
					\color{black}\begin{tabular*}{\textwidth}{@{\extracolsep{\fill}}lccccccccc@{\extracolsep{\fill}}}
						\toprule%
						\multicolumn{3}{@{}c@{}}{} & \multicolumn{2}{@{}c@{}}{$\mathcal{C}_H$} & \multicolumn{2}{@{}c@{}}{$\mathcal{P}_H$} &\multicolumn{2}{@{}c@{}}{$\mathcal{P}_\theta$} \\\cmidrule{4-5}\cmidrule{6-7}\cmidrule{8-9}%
						$n$     & $m+1$  & DoF       & Iter             & CPU         & Iter             & CPU  & Iter             & CPU   \\
						\midrule	
                        \multirow{4}{*}{$2^5$} 	&$	2^5	$&	30752	& 	107	& 	0.63	& 	11	& 	0.084  	& 	11	& 	0.084	\\
						&$	2^6	$&	127008	& 	141	& 	1.95	& 	11	& 	0.17	& 	12	& 	0.16	\\
						&$	2^7	$&	516128	& 	218	& 	7.91	& 	11	& 	0.57	& 	13	& 	0.64	\\
						&$	2^8	$&	2080800	& 	315	& 	62.44	& 	12	& 	3.82	& 	16	& 	4.87	\\
						\midrule	\multirow{4}{*}{$2^6$} 	&$	2^5	$&	61504	& 	106	& 1.00		& 	11	& 	0.13	& 	13	& 	0.15	\\
						&$	2^6	$&	254016	& 	154	& 	3.20	& 	11	& 	0.25	& 	13	& 	0.31	\\
						&$	2^7	$&	1032256	& 	219	& 	21.27	& 	13	& 	1.44	& 	14  	& 	1.49	\\
						&$	2^8	$&	4161600	& 	307	& 	155.67	& 	13	& 	8.18	& 	17	& 	10.69	\\
						\midrule	\multirow{4}{*}{$2^7$} 	&$	2^5	$&	123008	& 	107	& 	1.49	& 	13	& 	0.22	& 	13	& 	0.21	\\
						&$	2^6	$&	508032	& 	160	& 	6.30	& 	13	& 	0.58	& 	14	& 	0.63	\\
						&$	2^7	$&	2064512	& 	218	& 	47.22	& 	13	& 	3.44	& 	15	& 	4.55	\\
                        &$	2^8	$&	8323200	& 	303	& 	315.43	& 	13	& 	17.64	& 	18	& 	23.49	\\
                        \midrule	\multirow{4}{*}{$2^8$} 	&$	2^5	$&	246016	& 	118	& 	2.50	& 	14	& 	0.35	& 	15	& 	0.36	\\
						&$	2^6	$&	1016064	& 	177	& 	14.86	& 	14	& 	1.58	& 	15	& 	1.45	\\
						&$	2^7	$&	4129024	& 	220	& 	111.71	& 	14	& 	10.00	& 	17	& 	11.22	\\
                        &$	2^8	$&	16646400& 	299	& 	601.57	& 	15	& 	39.08	& 	19	& 	48.44	\\
						\botrule
					\end{tabular*}
				\end{minipage}
			\end{center}
		\end{table}

        \begin{table}[h]
			\begin{center}
				\begin{minipage}{\textwidth}
					\caption{\cred{Convergence results with MINRES for Example \ref{example_2D_Lin} when $\theta=0.5$ (the Crank-Nicolson method)}}\label{tab:case_2D_Lin_Crank}
					\color{black}\begin{tabular*}{\textwidth}{@{\extracolsep{\fill}}lccccccccc@{\extracolsep{\fill}}}
						\toprule%
						\multicolumn{3}{@{}c@{}}{} & \multicolumn{2}{@{}c@{}}{$\mathcal{C}_H$} & \multicolumn{2}{@{}c@{}}{$\mathcal{P}_H$} &\multicolumn{2}{@{}c@{}}{$\mathcal{P}_\theta$} \\\cmidrule{4-5}\cmidrule{6-7}\cmidrule{8-9}%
						$n$     & $m+1$  & DoF       & Iter             & CPU         & Iter             & CPU  & Iter             & CPU   \\
						\midrule	
                        \multirow{4}{*}{$2^5$} 	&$	2^5	$&	30752	& 	106	& 	0.61	& 	11	& 	0.085  	& 	11	& 	0.076	\\
						&$	2^6	$&	127008	& 	141	& 	1.94	& 	11	& 	0.19	& 	12	& 	0.18	\\
						&$	2^7	$&	516128	& 	216	& 	8.75	& 	11	& 	0.70	& 	13	& 	0.70	\\
						&$	2^8	$&	2080800	& 	309	& 	67.75	& 	12	& 	4.22	& 	17	& 	5.63	\\
						\midrule	\multirow{4}{*}{$2^6$} 	&$	2^5	$&	61504	& 	106	& 	0.96	& 	11	& 	0.12	& 	13	& 	0.14	\\
						&$	2^6	$&	254016	& 	154	& 	3.38	& 	11	& 	0.27	& 	13	& 	0.32	\\
						&$	2^7	$&	1032256	& 	218	& 	18.64	& 	11	& 	1.16	& 	14  	& 	1.42	\\
						&$	2^8	$&	4161600	& 	304	& 	164.44	& 	13	& 	9.68	& 	17	& 	11.38	\\
						\midrule	\multirow{4}{*}{$2^7$} 	&$	2^5	$&	123008	& 	107	& 	1.53	& 	13	& 	0.21	& 	13	& 	0.21	\\
						&$	2^6	$&	508032	& 	160	& 	6.78	& 	13	& 	0.63	& 	13	& 	0.66	\\
						&$	2^7	$&	2064512	& 	218	& 	48.20	& 	13	& 	4.12	& 	15	& 	4.63	\\
                        &$	2^8	$&	8323200	& 	301	& 	337.99	& 	13 & 	18.98	& 	19	& 	26.58	\\
                        \midrule	\multirow{4}{*}{$2^8$} 	&$	2^5	$&	246016	& 	117	& 	2.63	& 	14	& 	0.36	& 	15	& 	0.38	\\
						&$	2^6	$&	1016064	& 	177	& 	15.35	& 	14	& 	1.47	& 	15	& 	1.56	\\
						&$	2^7	$&	4129024	& 	221	& 	114.99	& 	14	& 	9.39	& 	17	& 	11.16	\\
                        &$	2^8	$&	16646400& 	299	& 	648.21	& 	15	& 	42.38	& 	19	& 	52.54	\\
						\botrule
					\end{tabular*}
				\end{minipage}
			\end{center}
		\end{table}

    \begin{table}[h]
			\begin{center}
				\begin{minipage}{\textwidth}
					\caption{\cred{Error results with MINRES for Example \ref{example_2D_Lin} }}\label{tab:case_2D_Lin_error}
					\color{black}\begin{tabular*}{\textwidth}{@{\extracolsep{\fill}}lccccccccc@{\extracolsep{\fill}}}
						\toprule%
						\multicolumn{3}{@{}c@{}}{} & \multicolumn{3}{@{}c@{}}{Backward Euler} & \multicolumn{3}{@{}c@{}}{Crank-Nicolson} \\\cmidrule{4-6}\cmidrule{7-9}%
						$n$     & $m+1$  & DoF       & $\mathcal{C}_H$   & $\mathcal{P}_H$  & $\mathcal{P}_\theta$  & $\mathcal{C}_H$   & $\mathcal{P}_H$  & $\mathcal{P}_\theta$\\
						\midrule	
                        \multirow{4}{*}{$2^5$} 	&$	2^5	$&	30752	& 	6.14e-4	& 	6.14e-4	& 	6.14e-4	& 	3.12e-6  	& 	3.12e-6	& 	3.12e-6	\\
						&$	2^6	$&	127008	& 	6.14e-4	& 	6.14e-4	& 	6.14e-4	& 	3.12e-6	& 	3.12e-6	& 	3.12e-6	\\
						&$	2^7	$&	516128	& 	6.14e-4	& 	6.14e-4	& 	6.14e-4	& 	3.12e-6	& 	3.12e-6	& 	3.12e-6	\\
						&$	2^8	$&	2080800	& 	6.14e-4	& 	6.14e-4	& 	6.14e-4	& 	3.12e-6	& 	3.12e-6	& 	3.12e-6	\\
						\midrule	\multirow{4}{*}{$2^6$} 	&$	2^5	$&	61504	& 	3.08e-4	& 	3.08e-4	& 	3.08e-4	& 	8.12e-7	& 	7.99e-7	& 	8.02e-7	\\
						&$	2^6	$&	254016	& 	3.08e-4	& 	3.08e-4	& 	3.08e-4	& 	8.05e-7	& 	8.01e-7	& 	8.03e-7	\\
						&$	2^7	$&	1032256	& 	3.08e-4	& 	3.08e-4	& 	3.08e-4	& 	8.15e-7	& 	8.07e-7  	& 	8.01e-7	\\
						&$	2^8	$&	4161600	& 	3.08e-4	& 	3.08e-4	& 	3.08e-4	& 	8.15e-7	& 	8.02e-7	& 	8.02e-7	\\
						\midrule	\multirow{4}{*}{$2^7$} 	&$	2^5	$&	123008	& 	1.54e-4	& 	1.54e-4	& 	1.54e-4	& 	2.38e-7	& 	1.99e-7	& 	2.04e-7	\\
						&$	2^6	$&	508032	& 	1.54e-4	& 	1.54e-4	& 	1.54e-4	& 	4.08e-7	& 	2.00e-7	& 	2.00e-7	\\
						&$	2^7	$&	2064512	& 	1.54e-4	& 	1.54e-4	& 	1.54e-4	& 	4.78e-7	& 	2.03e-7	& 	2.56e-7	\\
                        &$	2^8	$&	8323200	& 	1.54e-4	& 	1.54e-4	& 	1.54e-4	& 	4.10e-7	& 	7.08e-7	& 	5.87e-7	\\
                        \midrule	\multirow{4}{*}{$2^8$} 	&$	2^5	$&	246016	& 	7.71e-5	& 	7.71e-5	& 	7.71e-5	& 	1.79e-7	& 	5.79e-8	& 	4.98e-8	\\
						&$	2^6	$&	1016064	& 	7.71e-5	& 	7.71e-5	& 	7.71e-5	& 	6.37e-7	& 	5.91e-8	& 	6.93e-8	\\
						&$	2^7	$&	4129024	& 	7.71e-5	& 	7.71e-5	& 	7.71e-5	& 	6.50e-7	& 	1.53e-7	& 	2.58e-7	\\
                        &$	2^8	$&	16646400& 	7.71e-5	& 	7.71e-5	& 	7.71e-5	& 	3.76e-7	& 	6.42e-7	& 	7.60e-7	\\
						\botrule
					\end{tabular*}
				\end{minipage}
			\end{center}
		\end{table}
    \end{example}
    
    \begin{example} \label{example_3D}
        \cred{The last example is a three-dimensional problem of solving (\ref{eqn:heat}) with $\Omega=(0,1)^3$, $u_0(x,y)=x(x-1)y(y-1)z(z-1)$, $a(x,y,z)=10^{-3}$, $g=0$, $f=0$, and $T=1$.

        Tables \ref{tab:case_3D} and \ref{tab:case_3D_crank} respectively show the convergence results of MINRES when the backward Euler method and the Crank-Nicolson method are used for time. Overall, the preconditioners $\mathcal{P}_H$ and $\mathcal{P}_\theta$ still perform better than $\mathcal{C}_H$ with regard to iteration numbers. In this relatively well-conditioned case, all preconditioned solvers appear to have comparable CPU times.}
        
        \begin{table}[h]
			\begin{center}
				\begin{minipage}{\textwidth}
					\caption{\cred{Convergence results with MINRES for Example \ref{example_3D} when $\theta=1$ (the backward Euler method)}}\label{tab:case_3D}
					\color{black}\begin{tabular*}{\textwidth}{@{\extracolsep{\fill}}lccccccccc@{\extracolsep{\fill}}}
						\toprule%
						\multicolumn{3}{@{}c@{}}{} & \multicolumn{2}{@{}c@{}}{$\mathcal{C}_H$} & \multicolumn{2}{@{}c@{}}{$\mathcal{P}_H$} &\multicolumn{2}{@{}c@{}}{$\mathcal{P}_\theta$} \\\cmidrule{4-5}\cmidrule{6-7}\cmidrule{8-9}%
						$n$     & $m+1$  & DoF       & Iter             & CPU               & Iter             & CPU  & Iter             & CPU     \\
						\midrule
						\multirow{4}{*}{$2^3$} 	&$	2^3	$&	2744	& 	11	& 	0.034	& 	10	& 	0.027  & 	13	& 	0.029 	\\
						&$	2^4	$&	27000	& 	18	& 	0.12	& 	12	& 	0.093	& 	14	& 	0.10	\\
						&$	2^5	$&	238328	& 	21	& 	0.54	& 	13	& 	0.39	& 	16	& 	0.49	\\
						&$	2^6	$&	2000376	& 	21	& 	5.91	& 	13	& 	5.06	& 	16	& 	6.07	\\
						\midrule	\multirow{4}{*}{$2^4$} 	&$	2^3	$&	5488	& 	14	& 	0.032	& 	12	& 	0.037	& 	14	& 	0.040	\\
						&$	2^4	$&	54000	& 	18	& 	0.17	& 	15	& 	0.15	& 	17	& 	0.18	\\
						&$	2^5	$&	476656	& 	21	& 	1.08	& 	15	& 	0.85	& 	18  & 	0.96	\\
						&$	2^6	$&	4000752	& 	24	& 	15.77	& 	17	& 	13.32	& 	18	& 	14.01	\\
						\midrule	\multirow{4}{*}{$2^5$} 	&$	2^3	$&	10976	& 	14	& 	0.059	& 	14	& 	0.070	& 	15	& 	0.075	\\
						&$	2^4	$&	108000	& 	18	& 	0.29	& 	17	& 	0.27	& 	19	& 	0.31	\\
						&$	2^5	$&	953312	& 	22	& 	2.37	& 	18	& 	2.00	& 	21	& 	2.30	\\
                        &$	2^6	$&	8001504	& 	24	& 	32.41	& 	19	& 	30.75	& 	22	& 	36.00	\\
                        \midrule	\multirow{4}{*}{$2^6$} 	&$	2^3	$&	21952	& 	14	& 	0.084	& 	15	& 	0.094	& 	17	& 	0.13	\\
						&$	2^4	$&	216000	& 	18	& 	0.49	& 	18	& 	0.50	& 	21	& 	0.55	\\
						&$	2^5	$&	1906624	& 	22	& 	6.13	& 	21	& 	7.05	& 	24	& 	7.94	\\
                        &$	2^6	$&	16003008	& 	24	& 	65.54	& 	21	& 	66.41	& 	24	& 	75.87	\\
						\botrule
					\end{tabular*}
				\end{minipage}
			\end{center}
		\end{table}

        \begin{table}[h]
			\begin{center}
				\begin{minipage}{\textwidth}
					\caption{\cred{Convergence results with MINRES for Example \ref{example_3D} when $\theta=0.5$ (the Crank-Nicolson method)}}\label{tab:case_3D_crank}
					\color{black}\begin{tabular*}{\textwidth}{@{\extracolsep{\fill}}lccccccccc@{\extracolsep{\fill}}}
						\toprule%
						\multicolumn{3}{@{}c@{}}{} & \multicolumn{2}{@{}c@{}}{$\mathcal{C}_H$} & \multicolumn{2}{@{}c@{}}{$\mathcal{P}_H$} &\multicolumn{2}{@{}c@{}}{$\mathcal{P}_\theta$} \\\cmidrule{4-5}\cmidrule{6-7}\cmidrule{8-9}%
						$n$     & $m+1$  & DoF       & Iter             & CPU               & Iter             & CPU  & Iter             & CPU     \\
						\midrule
						\multirow{4}{*}{$2^3$} 	&$	2^3	$&	2744	& 	14	& 	0.040	& 	10	& 	0.034  & 	13	& 	0.030 	\\
						&$	2^4	$&	27000	& 	18	& 	0.10	& 	13	& 	0.095	& 	15	& 	0.11	\\
						&$	2^5	$&	238328	& 	21	& 	0.61	& 	13	& 	0.44	& 	17	& 	0.55	\\
						&$	2^6	$&	2000376	& 	21	& 	6.23	& 	13	& 	5.11	& 	17	& 	6.71	\\
						\midrule	\multirow{4}{*}{$2^4$} 	&$	2^3	$&	5488	& 	14	& 	0.034	& 	12	& 	0.037	& 	15	& 	0.048	\\
						&$	2^4	$&	54000	& 	18	& 	0.20	& 	15	& 	0.19	& 	17	& 	0.18	\\
						&$	2^5	$&	476656	& 	21	& 	1.10	& 	15	& 	0.90	& 	20  & 	1.16	\\
						&$	2^6	$&	4000752	& 	25	& 	17.49	& 	17	& 	13.98	& 	20	& 	16.36	\\
						\midrule	\multirow{4}{*}{$2^5$} 	&$	2^3	$&	10976	& 	14	& 	0.066	& 	14	& 	0.064	& 	16	& 	0.081	\\
						&$	2^4	$&	108000	& 	18	& 	0.30	& 	17	& 	0.30	& 	19	& 	0.35	\\
						&$	2^5	$&	953312	& 	23	& 	3.03	& 	18	& 	2.86	& 	22	& 	3.46	\\
                        &$	2^6	$&	8001504	& 	24	& 	34.45	& 	19	& 	32.64	& 	22	& 	37.47	\\
                        \midrule	\multirow{4}{*}{$2^6$} 	&$	2^3	$&	21952	& 	14	& 	0.10	& 	15	& 	0.095	& 	17	& 	0.11	\\
						&$	2^4	$&	216000	& 	18	& 	0.51	& 	18	& 	0.50	& 	22	& 	0.61	\\
						&$	2^5	$&	1906624	& 	23	& 	7.56	& 	21	& 	7.23	& 	25	& 	8.49	\\
                        &$	2^6	$&	16003008	& 	24	& 	69.35	& 	21	& 	69.73	& 	26	& 	85.77	\\
						\botrule
					\end{tabular*}
				\end{minipage}
			\end{center}
		\end{table}
    \end{example}

	\section{Conclusion}\label{sec:conclusions}
We have developed in this work a sine transform based preconditioning method which can be used for solving a wide range of time-dependent PDEs. Namely, our proposed \cred{preconditioners} can be applied for a large class of symmetrized all-at-once system $\mathcal{Y}\mathcal{\widehat{T}}$ resulted from discretizing the underlying equation. The method is generic, and can be used with high-order discretization schemes for time.

We have adopted the Krylov subspace solver of MINRES by exploiting the symmetry of the permuted matrix. In particular, we emphasize that our sine transform based \cred{preconditioners can be seen as optimal or quasi-optimal preconditioners, which are both symmetric positive definite} and efficiently implemented, according to the previous work on the spectral distribution of symmetrized Toeplitz matrix sequences \cite{MazzaPestana2018,Ferrari2019}.
We observe that a quasi optimal behavior is obtained in the case of multilevel coefficient matrices, which is a good result given the theoretical barriers proved in \cite{SeTy2} \cred{for matrix algebra preconditioners with equimodular transforms like multilevel circulants (see also \cite{nega-gen} for more discussion and more general results)}.

Also, specifically designed for MINRES, our \cred{preconditioners consistently outperform an existing block circulant preconditioner in both iteration counts and CPU times as shown in each numerical example. It is observed that they can achieve rapid convergence, especially in the ill-conditioned case. We stress that the severity of ill-conditioning of the original linear system mainly comes from both the spatial and temporal discretization schemes for the underlying equation. In addition to effective preconditioning, one remedy is to develop better discretization schemes in the first place for both space and time, which would be a potential research direction in the next stage. As long as the resulting linear systems from such discretization schemes are of lower triangular banded block Toeplitz structure, our proposed preconditioners can be incorporated. The required structure, mainly as a result of time discretization through an all-at-once process, is not difficult to obtain.}

Our preconditioning technique \cred{(i.e., $\mathcal{P}_{H}$)} utilizes the fast diagonalizability of the Laplacian matrix, which relies mainly on the uniform spatial grid setting. We however emphasize that such a uniform grid assumption is a common \cred{starting point} in the all-at-once preconditioning context, as with many existing works such as \cite{LIN2021110221,WU2021110076}. Along this line, it would be a direction for future research to develop efficient preconditioning methods when the physical domain is irregular - in this setting the theory of generalized locally Toeplitz sequences could be exploited for the spectral and convergence analysis \cite{GLT-LAA2,Ba}. \cred{Also, our proposed preconditioner $\mathcal{P}_{\theta}$, which does not require such fast diagonalizability and has shown excellent preconditioning effect in the numerical tests, deserves further investigation.}

For another future work, it would be interesting to combine the symmetrization MINRES solver with the block $\epsilon$-circulant preconditioning technique \cite{doi:10.1137/20M1316354,doi:10.1137/19M1309869,WU2021110076} which has only \cred{been} shown applicable for GMRES. Such a combination would further advance the pioneering block circulant preconditioning theory (cooperated with MINRES) originally proposed in \cite{doi:10.1137/16M1062016}, which is still at an early stage of development.

As already observed, a further interesting case is that of variable coefficients.  When $a(x)$ is not constant or there is also a dependency on time, that is $a \equiv a(x,t)$, the spectral analysis can be performed and the related solvers can be designed via multilevel {GLT tools (see \cite{book-GLT-II})}, by assuming that both $n$ and $m$ are allowed to tend to infinity. In this setting, instead of dealing with \cred{a} matrix valued symbol (whose size is the fixed parameter $m$), we have $d$ further levels in the structure, where $d$ is the spatial dimensionality of the problem, but with scalar-valued symbols. Of course, this setting is a challenge \cred{(see \cite{SeTy2,nega-gen} and references therein) to be considered in future work, even though the numerical tests considered in this work seem promising in this direction.}


\section*{Declarations}

\subsection*{Ethical Approval and Consent to participate} 
Not applicable.

\subsection*{Consent for publication} 
The authors consent to publication of this work in Numerical Algorithms of Springer Nature, when accepted.

\subsection*{Human and Animal Ethics} 
Not applicable.

\subsection*{Availability of supporting data} 
Data sharing is not applicable to this article as no datasets were generated or analyzed during the current study.

\subsection*{Competing interests} 
The authors declare no competing interests.

\subsection*{Funding}
The work of S. Hon was supported in part by the Hong Kong RGC under grant 22300921, a start-up allowance from the Croucher Foundation, and a Tier 2 Start-up Grant from Hong Kong Baptist University. The work of S. Serra-Capizzano was supported in part by INDAM-GNCS.
\cred{Furthermore, the work of Stefano Serra-Capizzano wss funded from the European High-Performance Computing Joint Undertaking  (JU) under grant agreement No 955701. The JU receives support from the European Union’s Horizon 2020 research and innovation programme and Belgium, France, Germany, Switzerland. Stefano Serra-Capizzano is also grateful for the support of the Laboratory of Theory, Economics and Systems – Department of Computer Science at Athens University of Economics and Business.}

\subsection*{Authors' contributions}
S. Hon developed the methodology, wrote the main manuscript text, and conducted the main numerical experiments. \cred{P. Y. Fung improved the numerical experiments and modified the manuscript text. J. Dong modified the manuscript text.} S. Serra-Capizzano validated the methodology and modified the manuscript text. All authors reviewed the manuscript.


	\bibliography{sn-bibliography}

\end{document}